\documentclass[11pt]{amsart}

\usepackage[margin=0.9in]{geometry}
\usepackage[utf8]{inputenc}
\usepackage[T1]{fontenc}
\usepackage{subfiles}
\usepackage{comment}
\usepackage{float}
\usepackage{marginnote}
\usepackage{euscript}
\usepackage[dvipsnames]{xcolor}
\usepackage{graphicx}
\usepackage{amssymb}
\usepackage{mathrsfs}
\usepackage{amsthm}
\usepackage{amsmath}
\usepackage{mathtools}
\usepackage[cal=cm]{mathalfa}
\usepackage{stmaryrd}
\usepackage{upgreek}
\usepackage{bbm}
\usepackage[breaklinks, hypertexnames=false, hyperfootnotes=false, colorlinks=true,linktocpage=true, allcolors=highlight]{hyperref}
\usepackage{cleveref}
\usepackage{url}
\usepackage{float}
\usepackage{todonotes} 
\usepackage[full]{textcomp}

\usepackage{tikz,tikz-cd,tikz-3dplot}
\usepackage{pgfplots}
\usetikzlibrary{calc}
\usetikzlibrary{fadings}
\usetikzlibrary{decorations.pathmorphing}
\usetikzlibrary{decorations.pathreplacing}
\usetikzlibrary{patterns}
\usetikzlibrary{arrows,shadows,positioning, calc, decorations.markings,
	hobby,quotes,angles,decorations.pathreplacing,intersections,shapes}
\usepgflibrary{shapes.geometric}
\usetikzlibrary{fillbetween,backgrounds}

\usepackage{xcolor}

\definecolor{highlight}{HTML}{0b5394}
\definecolor{faint}{HTML}{666666}

\usepackage{anyfontsize}
\usepackage[lining]{libertine}%
\usepackage{courier}
\usepackage[T1]{fontenc}
\usepackage[utf8]{inputenc}
\usepackage{csquotes}
\makeatletter
\renewcommand*\libertine@figurestyle{LF}
\makeatother
\usepackage[libertine,libaltvw,liby]{newtxmath}
\usepackage{microtype}

\usepackage{array}
\newcolumntype{H}{>{\setbox0=\hbox\bgroup}c<{\egroup}@{}}
\makeatletter
\def\@tocline#1#2#3#4#5#6#7{\relax
	\ifnum #1>\c@tocdepth 
	\else
	\par \addpenalty\@secpenalty\addvspace{#2}%
	\begingroup \hyphenpenalty\@M
	\@ifempty{#4}{%
		\@tempdima\csname r@tocindent\number#1\endcsname\relax
	}{%
		\@tempdima#4\relax
	}%
	\parindent\z@ \leftskip#3\relax \advance\leftskip\@tempdima\relax
	\rightskip\@pnumwidth plus4em \parfillskip-\@pnumwidth
	#5\leavevmode\hskip-\@tempdima
	\ifcase #1
	\or\or \hskip 1em \or \hskip 2em \else \hskip 3em \fi%
	#6\nobreak\relax
	\dotfill\hbox to\@pnumwidth{\@tocpagenum{#7}}\par
	\nobreak
	\endgroup
	\fi}
\makeatother

\tikzset{
	commutative diagrams/.cd,
	arrow style=tikz,
	diagrams={>=stealth}
}
\tikzset{
	arrow/.pic={\path[tips,every arrow/.try,->,>=#1] (0,0) -- +(0,4pt);},
	pics/arrow/.default={triangle 90}
}
\tikzset{->-/.style={decoration={
			markings,
			mark=at position .6 with {\arrow{latex}}},postaction={decorate}}
}
\tikzset{
	c/.style={every coordinate/.try}
}
\usepackage{enumerate}
\usepackage{enumitem}
\setlist[enumerate,1]{%
	label=(\roman*)	,itemsep=0.3em
	}
\setlist[itemize]{itemsep=0.3em}

\newcommand{\define}[1]{\textbf{#1}}

\newcommand{\mc}{\mathcal}
\newcommand{\mr}{\mathrm}
\newcommand{\mf}{\mathfrak}
\newcommand{\msf}{\mathsf}
\newcommand{\mbb}{\mathbb}

\DeclareMathOperator{\Aut}{Aut}

\DeclareMathOperator*{\Tot}{Tot}

\newcommand{\vir}{\mathrm{virt}}
\newcommand{\simple}{\mathrm{sim}}
\newcommand{\rubber}{\sim}

\newcommand{\Gm}{\mathbb{G}_{\mathsf{m}}}

\newcommand{\gpfont}[1]{#1}

\newcommand{\Aaff}[1]{\mathbb{A}^{\! #1}}
\newcommand{\hhh}{\mathsf{H}}
\newcommand{\cstar}{*}
\newcommand{\kfield}{\mathbbm{k}}
\newcommand{\Speck}{\operatorname{Spec}\kfield}

\newcommand{\modulifont}[1]{#1}
\newcommand{\Mbar}{\overline{\modulifont{M}}}

\newcommand{\Nbar}{\overline{\modulifont{N}}}
\newcommand{\Fbar}{\overline{\modulifont{F}}}

\newcommand{\invariantfont}[1]{\mathsf{#1}}
\newcommand{\LGW}[3]{\invariantfont{LGW}_{#1}(#2)}
\newcommand{\OGW}[2]{\invariantfont{OGW}_{#1}(#2)}

\newcommand{\ri}{\mathrm{i}}
\newcommand{\re}{\mathrm{e}}
\newcommand{\rt}{\mathrm{t}}

\newcommand{\vertspace}{\hspace{0.2ex}|\hspace{0.2ex}}
\newcommand{\torX}{X}
\newcommand{\defS}{S^{\prime}}
\newcommand{\defDone}{D_1^{\prime}}

\newcommand{\qbinom}{\genfrac{[}{]}{0pt}{}}
\makeatletter 
\def\makeCal#1{%
	\expandafter\newcommand\csname c#1\endcsname{\mathcal{#1}}}

\def\makeBB#1{%
	\expandafter\newcommand\csname b#1\endcsname{\mathbb{#1}}}

\def\makeFrak#1{%
	\expandafter\newcommand\csname f#1\endcsname{\mathfrak{#1}}}

\def\makeScr#1{%
	\expandafter\newcommand\csname s#1\endcsname{\mathscr{#1}}}

\def\makeSF#1{%
	\expandafter\newcommand\csname sf#1\endcsname{\mathsf{#1}}}

\count@=0
\loop
\advance\count@ 1
\edef\y{\@Alph\count@} 
\expandafter\makeCal\y
\expandafter\makeBB\y
\expandafter\makeFrak\y
\expandafter\makeScr\y
\expandafter\makeSF\y
\ifnum\count@<26
\repeat

\theoremstyle{plain}
\newtheorem{thm}{Theorem}[section]
\newtheorem{lem}[thm]{Lemma}

\newtheorem{prop}[thm]{Proposition}
\newtheorem{conj}[thm]{Conjecture}

\newtheorem*{conj*}{Conjecture}

\newtheorem*{cor*}{Corollary}

\theoremstyle{definition}
\newtheorem{defn}[thm]{Definition}
\newtheorem{rmk}[thm]{Remark}
\newtheorem{notn}[thm]{Notation}

\newtheorem{construction}[thm]{Construction}

\theoremstyle{plain}
\newtheorem{innercustomthm}{Theorem}
\newenvironment{customthm}[1]
{\renewcommand\theinnercustomthm{#1}\innercustomthm}
{\endinnercustomthm}

\theoremstyle{plain}
\newtheorem{innercustomcor}{Corollary}
\newenvironment{customcor}[1]
{\renewcommand\theinnercustomcor{#1}\innercustomcor}
{\endinnercustomcor}

\theoremstyle{plain}
\newtheorem{innercustomconj}{Conjecture}
\newenvironment{customconj}[1]
{\renewcommand\theinnercustomconj{#1}\innercustomconj}
{\endinnercustomconj}

\theoremstyle{definition}
\newtheorem{innercustomassumption}{Assumption}
\newenvironment{customassumption}[1]
{\renewcommand\theinnercustomassumption{#1}\innercustomassumption}
{\endinnercustomassumption}

\theoremstyle{plain}
\newtheorem{innerpracticalresult}{Practical Result}
\newenvironment{practicalresult}[1]
{\renewcommand\theinnerpracticalresult{#1}\innerpracticalresult}
{\endinnerpracticalresult}

\crefname{equation}{Eq.}{Eqs.}
\crefname{eqnarray}{Eq.}{Eqs.}
\crefname{algo}{algorithm}{algorithms}
\crefname{conj}{conjecture}{conjectures}
\crefname{lem}{lemma}{lemmas}
\crefname{thm}{theorem}{theorems}
\crefname{claim}{claim}{claims}
\crefname{rmk}{remark}{remarks}
\crefname{prop}{proposition}{propositions}
\crefname{section}{section}{sections}
\crefname{appendix}{appendix}{appendices}
\crefname{cor}{corollary}{corollaries}
\crefname{figure}{figure}{figures}
\crefname{table}{table}{tables}
\crefname{example}{example}{examples}
\crefname{prob}{problem}{problems}
\crefname{assm}{assumption}{assumptions}
\crefname{defn}{definition}{definitions}
\crefname{speculation}{speculation}{speculations}
\crefname{construction}{construction}{constructions}
\crefname{innercustomthm}{theorem}{theorems}
\crefname{innercustomconj}{conjecture}{conjectures}
\crefname{innercustomassumption}{assumption}{Assumption}
\crefname{innerpracticalresult}{practical result}{practical results}

\usepackage[
		backend=biber,
		style=alphabetic,
		natbib=true,
		url=false,
		doi=false,
		eprint=true,
		isbn=false,
		maxcitenames=5,
		maxbibnames=4,
		maxalphanames=4,
		useprefix=true
	]{biblatex}

\renewbibmacro{in:}{%
	\ifboolexpr{%
		test {\ifentrytype{article}}%
		or
		test {\ifentrytype{inproceedings}}%
	}{}{\printtext{\bibstring{in}\intitlepunct}}%
}

\DeclareFieldFormat[article,inbook,incollection,inproceedings,patent,thesis,unpublished,techreport,misc,book]{title}{\mkbibquote{#1}}

\title{The log-open correspondence for two-component Looijenga pairs}
\author{Yannik Schuler}

\address{University of Sheffield, School of Mathematics and Statistics, Hounsfield Road,  Sheffield S3 7RH, United Kingdom.}
\address{University of Cambridge, Department of Pure Mathematics and Mathematical Statistics, Wilberforce Road, Cambridge CB3 0WB, United Kingdom}
\email{ys667@cam.ac.uk}

\begin{document}

\begin{abstract}
	A two-component Looijenga pair is a rational smooth projective surface with an anticanonical divisor consisting of two transversally intersecting curves. We establish an all-genus correspondence between the logarithmic Gromov--Witten theory of a two-component Looijenga pair and open Gromov--Witten theory of a toric Calabi--Yau threefold geometrically engineered from the surface geometry. This settles a conjecture of Bousseau, Brini and van Garrel in the case of two boundary components. We also explain how the correspondence implies BPS integrality for the logarithmic invariants and provides a new means for computing them via the topological vertex method.
\end{abstract}

\maketitle
\vspace*{-2em}
\setcounter{tocdepth}{1}
\hypersetup{bookmarksdepth = 3}
\tableofcontents

\vspace*{-2em}
\section*{Introduction}
In \cite{BBvG2} Bousseau, Brini and van Garrel discovered a surprising relationship between two at first sight quite different curve counting theories:
\begin{itemize}
	\item \hyperref[sec: log GW setup]{\textcolor{black}{\textbf{Logarithmic Gromov--Witten theory}}} of a Looijenga pair
	\begin{equation*}
		(S  \vertspace  D_1+\ldots +D_l)
	\end{equation*}
	which is a rational smooth projective surfaces $S$ together with an anticanonical singular nodal curve $D_1+\ldots + D_l$. To this datum one can associate a Gromov--Witten invariant
	\begin{equation*}
		\LGW{g,\mathbf{\hat{c}},\upbeta}{S  \vertspace  D_1+\ldots +D_l}{(-1)^g \uplambda_g \Pi_{i=1}^{l-1} \mr{ev}_i^{\cstar} (\mr{pt})} \in \bQ
	\end{equation*}
	enumerating genus $g$, class $\upbeta$ stable logarithmic maps to $(S  \vertspace  D_1+\ldots +D_l)$ with maximum tangency along each irreducible component $D_i$. Additionally, there are $l-1$ interior markings with a point condition and an insertion of the top Chern class of the Hodge bundle. The tangency order of the markings along $D_1,\ldots,D_l$ is recorded in a matrix $\mathbf{\hat{c}}$.

	\vspace{0.3em}

	\item \hyperref[sec: open GW setup]{\textcolor{black}{\textbf{Open Gromov--Witten theory}}} of a toric triple
	\begin{equation*}
		\big(X,(L_i,\msf{f}_i)_{i=1}^k\big)
	\end{equation*}
	consisting of the toric Calabi--Yau threefold $X$ and a collection of framed Aganagic--Vafa Lagrangian submanifolds $(L_i,\msf{f}_i)$. To this data one can associate open Gromov--Witten invariants
	\begin{equation*}
		\OGW{g,((w_1), \ldots, (w_k)),\upbeta'}{X,(L_i,\msf{f}_i)_{i=1}^k} \in \bQ
	\end{equation*}
	enumerating genus $g$, class $\upbeta'$ stable maps to $X$ from Riemann surfaces with $k$ boundaries, each wrapping one of the Lagrangian submanifolds $L_i$ exactly $w_i$ times.
\end{itemize}

It was conjectured in \cite{BBvG2} that starting from a Looijenga pair one can geometrically engineer a toric triple so that the above two curve counts turn out to be essentially equal.

\begin{customconj}{A}
	\label{conj: BBvG log-open}
	\cite[Conjecture~1.3]{BBvG2} Let $(S  \vertspace  D_1+\ldots +D_l)$ be a Looijenga pair and $\upbeta$ an effective curve class in $S$ satisfying some technical conditions. Then there exists a toric triple $(X,(L_i,\mathsf{f}_i)_{i=1}^{l-1})$ and an effective curve class $\beta'$ in $X$ such that
	\begin{equation}
		\label{eq: BBvG log-open}
			\sum_{g\geq 0} \hbar^{2g-2} ~ \OGW{g,\mathbf{c},\upbeta'}{X, (L_i,\msf{f}_{i})_{i=1}^{l-1}} = \bigg(\prod_{i=1}^{l-1} \frac{(-1)^{D_i \cdot \upbeta}}{D_i \cdot \upbeta} \bigg) \, \frac{(-1)^{D_l \cdot \upbeta +1}}{2 \sin \tfrac{(D_l \cdot \upbeta)\hbar}{2}} ~ \sum_{g\geq 0} \hbar^{2g-1} ~ \LGW{g,\mathbf{\hat{c}},\upbeta}{S  \vertspace  D_1+\ldots+D_l}{(-1)^g \uplambda_g \, \Pi_{i=1}^{l-1} \mr{ev}_i(\mr{pt}) }
	\end{equation}
	where $\mathbf{c}$ is obtained from the contact datum $\mathbf{\hat{c}}$ by deleting the $l-1$ interior markings and the marking with tangency along $D_l$.
\end{customconj}

For $l=1$ the right-hand side of \eqref{eq: BBvG log-open} reduces to the generating series of ordinary Gromov–Witten invariants of $X$ in which case the conjecture has already been proven in \cite[Theorem D]{GNS23:bicyclic}. In this paper we investigate the case of two-component Looijenga pairs $(S  \vertspace  D_1+D_2)$ and as another application of the main result in \cite{GNS23:bicyclic} we will establish \Cref{conj: BBvG log-open} for these geometries.


\subsection{Logarithmic-Open Correspondence}
\label{sec: intro log-open}
Consider a two-component Looijenga pair $(S \vertspace D_1+D_2)$ together with an effective curve class $\upbeta$ in $S$. We will impose
\begin{customassumption}{$\star$}
	\label{assumption: star}
	$(S \vertspace D_1+D_2)$ and $\upbeta$ satisfy
	\begin{itemize}
		\item $D_i \cdot \upbeta >0$, $i\in\{1,2\}$,
		\item $D_2 \cdot D_2\geq 0$,
		\item $(S\vertspace D_1)$ deforms into a pair $(\defS \vertspace \defDone)$ with $\defS$ a smooth projective toric surface and $\defDone$ a toric hypersurface.\footnote{We expand on the last condition in \Cref{rmk: explanation deformation property}.}
	\end{itemize}
\end{customassumption}
Moreover, we denote by $\mathbf{\hat{c}}$ the following contact datum along $D_1+D_2$
\begin{equation*}
	\mathbf{\hat{c}} = \bigg(\raisebox{-0.7ex}{\begin{minipage}{13.2em}$\hspace{0.2em}\begin{matrix}
				0 & \cdots & 0 & c_1 & \mathclap{\overbrace{\makebox[4.5em]{$\cdots$}}^{= \, \mathbf{c}}} & c_n & 0\\
				0 & \mathclap{\underbrace{\makebox[4.5em]{$\cdots$}}_{\text{$m$ times}}} & 0 & 0 & \cdots & 0 & D_2 \cdot \upbeta
			\end{matrix}$\end{minipage}}\bigg)
\end{equation*}
where we assume $n\geq m$ and $c_i>0$ for all $i\in \{1,\ldots,n\}$. This means there are $m$ interior markings, $n$ markings tangent to $D_1$ and an additional marking with maximum tangency along $D_2$. We set
\begin{equation}
	\label{eq: intro log GW invariant}
	\LGW{g,\mathbf{\hat{c}},\upbeta}{S \vertspace D_1+D_2}{ (-1)^g \uplambda_g \, \Pi_i \mr{ev}_{i=1}^n(\mr{pt})} \coloneqq \int_{[\Mbar_{g,\mathbf{\hat{c}},\upbeta}(S \vertspace D_1 + D_2)]^{\vir}} (-1)^g \uplambda_g \, \prod_{i=1}^{m} \mr{ev}_i^{\cstar}\big[\mr{pt}_S\big] ~ \prod_{j=m+1}^{n} \mr{ev}_j^{\cstar}\big[\mr{pt}_{D_1}\big]\,.
\end{equation}
Especially, note that we impose a point condition at all $m$ interior markings. Our main result is the following logarithmic to open correspondence.

\begin{customthm}{B} \label{thm: intro log-open} (\Cref{thm: log open})
	Assuming \labelcref{assumption: star} there is a toric triple $(X,L,\mathsf{f})$ (see \Cref{constr: toric triple}) and a curve class $\upbeta^{\prime}$ in $X$ satisfying
	\begin{equation}
		\label{eq:log-open intro}\sum_{g\geq 0} \hbar^{2g-2} ~ \OGW{g,\mathbf{c},\upbeta^{\prime}}{X,L,\msf{f}} = \frac{(-1)^{D_1 \cdot \upbeta}}{m! \, \textstyle{\prod_{i=0}^{m-1}} c_{n-i}} \, \frac{(-1)^{D_2 \cdot \upbeta +1}}{2 \sin \tfrac{(D_2 \cdot \upbeta)\hbar}{2}} ~ \sum_{g\geq 0} \hbar^{2g-1} ~ \LGW{g,\mathbf{\hat{c}},\upbeta}{S \vertspace D_1+D_2}{(-1)^g \uplambda_g \,\Pi_{i=1}^{m} \mr{ev}_i^{\cstar}\big[\mr{pt}_S\big] ~ \Pi_{j=m+1}^{n} \mr{ev}_j^{\cstar}\big[\mr{pt}_{D_1}\big]} \,.
	\end{equation}
\end{customthm}

The specialisation of the above correspondence to the case $m=n=1$ immediately gives

\begin{customcor}{C}
	\label{cor: C}
	\Cref{conj: BBvG log-open} holds for all two-component Looijenga pairs ($l=2$) satisfying \Cref{assumption: star}.
\end{customcor}

\subsection{Topological vertex}
Computing Gromov--Witten invariants of logarithmic Calabi--Yau surfaces is usually a rather tedious endeavour. In higher genus with the presence of a $\uplambda_g$ insertion the only general tools available are scattering diagrams and tropical correspondence theorems \cite{Bou19:TropRefCurveCounting,Bou20:QuTropVert,Bou21:RefinedFloorDiags,KSK23:TropRefCurveCountDesc}.

As an application of \Cref{thm: intro log-open} we obtain a new --- highly efficient --- means for computing logarithmic Gromov Witten invariants of Looijenga pairs. Since open Gromov--Witten invariants can be computed using the topological vertex \cite{AKMV05:TopVert,LLLZ09:MathTopVert}, the same method can be applied to determine the opposite side of our correspondence theorem.

\begin{practicalresult}{D} (\Cref{sec: Top Vert}) \label{result: Top Vert}
	The logarithmic Gromov--Witten invariants \eqref{eq: intro log GW invariant} of a two-component Looijenga pair satisfying \Cref{assumption: star} are computed by the topological vertex.
\end{practicalresult}

\subsection{BPS integrality}
In general, Gromov--Witten invariants are rational numbers. However, in many cases they are expected to exhibit underlying integral BPS type invariants. For open Gromov--Witten invariants this behaviour was first observed by Labastida--Mari\~{n}o--Ooguri--Vafa \cite{OV00:KnotInvTopStr,LM01:PolyInvTorKnotTopStr,LMV00:KnotsLinksBranes,MV02:FramedKnotsLargeN}, was studied extensively in explicit examples by Luo--Zhu \cite{LZ19:LMOVunknot,Zhu19:TopStrQUivVarRRid,Zhu22:IntTopStrQuTwoFct} and got recently proven by Yu \cite{Yu23:BPS} for general toric targets. Hence, as a corollary of our main result \mbox{\Cref{thm: intro log-open}} we find that logarithmic invariants feature the same property.

\begin{customthm}{E} (\Cref{thm: BPS integrality}) \label{thm: intro BPS integrality}
	The logarithmic Gromov--Witten invariants \eqref{eq: intro log GW invariant} of a two-component Looijenga pair satisfying \Cref{assumption: star} exhibit BPS integrality.
\end{customthm}

In \Cref{sec: BPS integrality} we spell this statement out in more detail. We conjecture that the above observation holds without imposing \Cref{assumption: star} as well (\Cref{conj: BPS integrality}).

\subsection{Strategy}
The proof of \Cref{thm: intro log-open} splits into several steps which partly have already been carried out elsewhere:
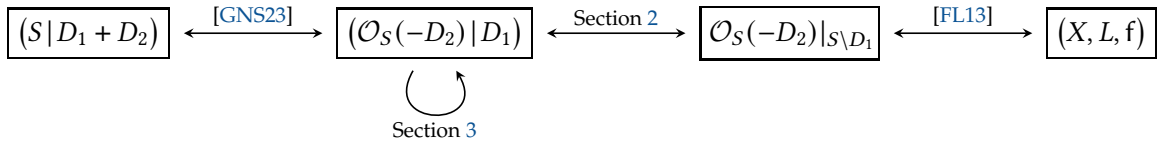
\begin{equation*}
	\begin{tikzcd}
		\fbox{$\big(S \vertspace  D_1+D_2\big)$} \ar[rr,leftrightarrow,"\text{\cite{GNS23:bicyclic}}"] & & \fbox{$\big(\cO_S(-D_2) \vertspace  D_1\big)$} \ar[out=240,in=300,loop,looseness=4,swap,"\text{\Cref{sec: reduction}}"] \ar[rr,leftrightarrow, "\text{\Cref{sec: log 3fold to open proof}}"] & & \fbox{$\cO_S(-D_2)\lvert_{S \setminus D_1}\phantom{\big|\hspace{-0.8ex}}$} \ar[rr,leftrightarrow, "\text{\cite{FL13:OpenGWtorCY3}}"] & & \fbox{$\big(X,L,\mathsf{f}\big)$} \\
	\end{tikzcd}
\end{equation*}
In \cite{GNS23:bicyclic} van Garrel, Nabijou and the author prove a comparison statement between the Gromov--Witten theory of $(S \vertspace D_1+D_2)$ and $(\cO_S(-D_2)  \vertspace  D_1)$. Together with a result of Fang and Liu \cite{FL13:OpenGWtorCY3} expressing the open Gromov--Witten invariants of a toric triple $(X,L,\mathsf{f})$ as descendant invariants of $X$ it is therefore sufficient to prove a formula for relative Gromov--Witten invariants of $(\cO_S(-D_2)  \vertspace  D_1)$ in terms of descendant invariants of $X=\Tot\cO_S(-D_2)|_{S \setminus D_1}$ in order to establish \Cref{thm: intro log-open}. We will prove such a relation in \Cref{sec: log 3fold to open proof} in case there are no interior markings ($m=0$). Later in \Cref{sec: reduction} we will reduce the general case ($m\geq 0$) to the one proven earlier.

\subsection{Context \& Prospects}

\subsubsection*{Stable maps to Looijenga pairs}
\Cref{conj: BBvG log-open} was formulated by Bousseau, Brini and van Garrel in \cite{BBvG2} motivated by a direct calculation and comparison of both sides of the correspondence. More precisely, in loc.~cit.~\Cref{conj: BBvG log-open} was established for all so called tame Looijenga pairs (up to a possible relabelling of boundary components) and later in \cite{Kra21:qBinomProof,BS23:Quasitame} the list of examples was extended to encompass all quasi-tame pairs as well. When $l=2$ a Looijenga pair $(S\vertspace D_1 + D_2)$ is called tame if $D_i\cdot D_i>0$, $i\in\{1,2\}$, and quasi-tame if $\Tot \cO_{S}(-D_1) \oplus \cO_{S}(-D_2)$ deforms to the total space of line bundles associated to a tame pair. It turns out there is only a finite number of such quasi-tame Looijenga pairs \cite[Section~2.4]{BBvG2} and for $l=2$ one can check that apart from $\bF_0(2,2)$ all of them satisfy \Cref{assumption: star}.\footnote{Indeed, examining the list of all two-component quasi-tame Looijenga pairs \cite[Table 1]{BBvG2} one observes that all of them satisfy the second property of \Cref{assumption: star}. The third property is more subtle but we see that apart from $\bF_0(2,2)$, $\mr{dP}_3(1,1)$ and $\mr{dP}_3(0,2)$ one can indeed equip the surface with a $\Gm^2$-action preserving $D_1$. In \Cref{sec:dP3 0 2} we explain that $\mr{dP}_3(0,2)$ also satisfies the third property of \Cref{assumption: star} by deforming the centre of the blowup. The pair $\mr{dP}_3(1,1)$ can be treated similarly. However, regarding the remaining case $\bF_0(2,2)$ the author does not know if and how the geometry can be treated using the methods developed in this paper.} Thus, regarding two-component Looijenga pairs \Cref{cor: C} covers all cases previously known in the literature with the exception of $\bF_0(2,2)$. Among these is the quasi-tame pair $\mr{dP}_3(0,2)$. For this geometry an explicit formula for its all-genus generating series of logarithmic Gromov--Witten invariants was found in \cite{BS23:Quasitame} by proving an intricate identity of $q$-hypergeometric functions. As an application of our main result we are going to reprove this formula in \Cref{sec:dP3 0 2} by using the topological vertex.



There is certainly the question whether the techniques that went into the proof of \Cref{thm: intro log-open} can be generalised to cover Looijenga pairs with $l>2$ boundary components as well. The most difficult part in this endeavour will most likely be to find an appropriate generalisation of \cite[Theorem A]{GNS23:bicyclic}. Moreover, in the case $l\geq 3$ the construction of the toric triple proposed in \cite[Example 6.5]{BBvG2} involves a flop which numerically does the right job but needs to be understood geometrically better first in order to generalise the approach of this paper.

\subsubsection*{Open/Closed duality of Liu--Yu}
Liu and Yu \cite{LY22:OpenClosed,LY22:OpenClosedOrbi} prove a correspondence between genus zero open Gromov--Witten invariants of a toric threefold and local Gromov--Witten invariants of some Calabi--Yau fourfold constructed from the threefold. If we combine \Cref{thm: intro log-open} with \cite[Theorem B]{GNS23:bicyclic} we obtain a similar relation between open/local Gromov--Witten invariants of
\begin{equation}
	\label{eq: open cloed comparison}
	\begin{tikzcd}
		\fbox{$(X,L,\msf{f})$} \ar[rr,leftrightarrow] & & \fbox{$\cO_S(-D_1) \oplus \cO_S(-D_2) $}\,.
	\end{tikzcd}
\end{equation}
It is a combinatorial exercise to check that starting from the toric triple $(X,L,\msf{f})$ as we define in \Cref{constr: toric triple} the fourfold constructed in \cite[Section~2.4]{LY22:OpenClosedOrbi} is indeed deformation equivalent (in the sense of \Cref{rmk: explanation deformation property}) to the right-hand side of \eqref{eq: open cloed comparison}.\footnote{The key observation is that the toric divisor $D$ we are deleting in \Cref{constr: toric triple} is reinserted in Liu--Yu's construction \cite[Section~2.4]{LY22:OpenClosedOrbi} of the fourfold as the divisor associated to the ray $\widetilde{b}_{R+1} \bR_{\geq 0}$.} It should, however, be stressed that the correspondence of Liu and Yu holds in a more general context than just local surfaces.

\subsubsection*{BPS invariants}
Motivated by \Cref{thm: intro BPS integrality} we conjecture BPS integrality for Gromov--Witten invariants of type \eqref{eq: intro log GW invariant} for two-component logarithmic Calabi--Yau surfaces not necessarily satisfying \Cref{assumption: star} (\Cref{conj: BPS integrality}). This prediction overlaps with a conjecture of Bousseau \cite[Conjecture 8.3]{Bou20:QuTropVert} in special cases. So one may wonder about a simultaneous generalisation of both conjectures. Especially, Bousseau's conjecture also covers logarithmic Calabi--Yau surfaces with more than two components. A geometric proof of \Cref{conj: BBvG log-open} in the case of $l>2$ components in combination with LMOV integrality \cite{Yu23:BPS} might allow progress in this direction.

\subsection*{Acknowledgements}
First and foremost I would like to thank my supervisor Andrea Brini for suggesting this project, innumerous discussions and his continuous support. Special thanks are also due to Cristina Manolache for clarifying discussions concerning key steps in the proof of our main result. I also benefited from discussions with Alberto Cobos R\'{a}bano, Michel van Garrel, Samuel Johnston, Navid Nabijou, Dhruv Ranganathan, Ajith Urundolil Kumaran and Song Yu which I am highly grateful for.

Parts of this work have been carried out during research visits at the University of Cambridge. I thank the university and St John's College for hospitality and financial support.

\section{Statement of the main result}

\subsection{Logarithmic Gromov--Witten theory}
\label{sec: log GW setup}
\begin{defn}
	\label{defn: l comp Looijenga pair}
	An \define{$l$-component Looijenga pair} $(S \vertspace D_1 + \ldots +D_l)$ is a rational smooth projective surface $S$ together with an anticanonical singular nodal curve $D_1 + \ldots + D_l$ with $l$ irreducible components $D_1,\ldots,D_l$.
\end{defn}
Given a two-component Looijenga pair $(S \vertspace D_1 + D_2)$, we denote by
\begin{equation*}
	\Mbar_{g,\mathbf{\hat{c}},\upbeta}(S \vertspace D_1 + D_2)
\end{equation*}
the moduli stack of genus $g$, class $\upbeta$ stable logarithmic maps to $(S \vertspace D_1 + D_2)$ with markings whose tangency along $D_1 + D_2$ is encoded in $\mathbf{\hat{c}}$ \cite{Chen14:StabLogMapsDFpairs,AC14:StabLogMapsDFpairsII,GS13:LogGWInv}. We fix this tangency data to be
\begin{equation*}
	\mathbf{\hat{c}} = \bigg(\raisebox{-0.9em}{\begin{minipage}{13.2em}$\hspace{0.2em}\begin{matrix}
				0 & \cdots & 0 & c_1 & \cdots & c_n & 0\\
				0 & \mathclap{\underbrace{\makebox[4.5em]{$\cdots$}}_{\text{$m$ times}}} & 0 & 0 & \cdots & 0 & D_2 \cdot \upbeta
			\end{matrix}$\end{minipage}}\bigg)
\end{equation*}
for some $c_1,\ldots,c_n> 0$ with $\sum_i c_i = D_1 \cdot \upbeta$. This means there are $m$ interior markings (no tangency), $n$ markings with tangency along $D_1$ and a marking having maximum tangency with $D_2$. The virtual dimension of the moduli stack is $g+m+n$ and so we define
\begin{equation*}
	\LGW{g,\mathbf{\hat{c}},\upbeta}{S \vertspace D_1+D_2}{(-1)^g \uplambda_g \, \upgamma} \coloneqq \int_{[\Mbar_{g,\mathbf{\hat{c}},\upbeta}(S \vertspace D_1 + D_2)]^{\vir}} (-1)^g \uplambda_g \, \prod_{i=1}^{m} \mr{ev}_i^{\cstar}\big[\mr{pt}_S\big] ~ \prod_{j=m+1}^{n} \mr{ev}_j^{\cstar}\big[\mr{pt}_{D_1}\big]
\end{equation*}
where $\uplambda_g$ is the top Chern class of the Hodge bundle.

The ultimate goal of this section is to relate the above invariants to the ones of an open subset of (a deformation of) the threefold $Y \coloneqq \Tot \cO_{S}(-D_2)$. Let us recall the main result of \cite{GNS23:bicyclic} which serves as a first step towards this direction. Writing $\uppi$ for the projection $Y\rightarrow S$, we set $D \coloneqq \uppi^{-1}( D_1)$ and define
\begin{equation}
	\label{eq: def log GW Y D}
	\LGW{g,\mathbf{\tilde{c}},\upbeta}{Y \vertspace D}{\upgamma} \coloneqq \int_{[\Mbar_{g,\mathbf{\tilde{c}},\upbeta}(Y \vertspace D)]^{\vir}} \prod_{i=1}^{m} \mr{ev}_i^{\cstar}\big(\uppi^{\cstar}\big[\mr{pt}_S\big]\big) ~ \prod_{j=m+1}^{n} \mr{ev}_j^{\cstar}\big(\uppi^{\cstar}\big[\mr{pt}_{D_1}\big]\big)
\end{equation}
where $\mathbf{\tilde{c}}$ is obtained from $\mathbf{\hat{c}}$ by deleting the marking with tangency along $D_2$.

\begin{thm}
	\label{thm: GNS main theorem}
	\cite[Theorem A]{GNS23:bicyclic} Suppose $D_2 \cdot D_2 \geq 0$ and $\upbeta$ is so that $D_1\cdot \upbeta \geq 0$ and  $D_2 \cdot \upbeta > 0$. Then
	\begin{equation}
		\label{eq: GNS main theorem}
		\sum_{g\geq 0} \hbar^{2g-2} ~\LGW{g,\mathbf{\tilde{c}},\upbeta}{Y \vertspace D}{\upgamma} = \frac{(-1)^{D_2 \cdot \upbeta+1}}{2 \sin \tfrac{(D_2 \cdot \upbeta)\hbar}{2}} \sum_{g\geq 0} \hbar^{2g-1} ~ \LGW{g,\mathbf{\hat{c}},\upbeta}{S \vertspace D_1+D_2}{(-1)^g \uplambda_g \, \upgamma}\,.
	\end{equation}
\end{thm}

\subsection{Open Gromov--Witten theory}
\label{sec: open GW setup}
Open Gromov--Witten invariants are a delicate topic. We will work in the framework of Fang and Liu \cite{FL13:OpenGWtorCY3} which builds on ideas of \cite{KL02:StabMapsBoundary,LS02:OpStrInstRelStabM}. We keep details to a minimum since at no point we will be working with open Gromov--Witten invariants directly. We refer the interested reader to \cite{FL13:OpenGWtorCY3,KL02:StabMapsBoundary} for more details.

\subsubsection{Preliminaries on toric Calabi--Yau threefolds}
\label{sec: open geom setup}
Let $\torX$ be a smooth toric Calabi--Yau threefold. We will write $\hat{\gpfont{T}}\cong \Gm^3$ for its dense torus and $\gpfont{T}\cong \Gm^2$ for its \define{Calabi--Yau torus}. The latter is defined as the kernel $\gpfont{T}=\ker \upchi$ of the character $\upchi$ associated to the induced $\hat{\gpfont{T}}$-action on $K_{\torX}\cong \cO_{\torX}$.

If $\torX$ can be realised as a symplectic quotient, which is the case when $\torX$ is semiprojective \cite{HS02:TorHKV}, Aganagic and Vafa \cite{AV00:MSDbranesHolDiscs} construct a certain class of Lagrangian submanifolds $L$ in $\torX$. Relevant for us here is the fact that these Lagrangian submanifolds are invariant under the action of the maximal compact subgroup $\gpfont{T}_{\mbb{R}} \hookrightarrow T$ and are homeomorphic to $\mbb{R}^{\!2} \times S^1$. Moreover, these Lagrangian submanifolds intersect a unique one dimensional torus orbit closure of $\torX$ in a circle $S^1$. For short we will call a Lagrangian submanifold $L$ of the type considered by Aganagic and Vafa an \define{outer brane} if it intersects a non-compact one dimensional torus orbit closure.

Let us write $\upmu$ for the moment map
\begin{equation*}
	\upmu : \torX \longrightarrow \mr{Lie}(\gpfont{T}_{\mbb{R}})^{\cstar}
\end{equation*}
and fix an isomorphism $\mr{Lie}(\gpfont{T}_{\mbb{R}})^{\cstar} \cong \mbb{R}^2$. We call the image of the union of all one dimensional torus orbit closures under $\upmu$ the \define{toric diagram} of $\torX$. It is an embedded trivalent graph since there are always three $\hat{\gpfont{T}}$-preserved one dimensional strata meeting in a torus fixed point (see \Cref{fig: example outer brane}). Now let $L$ be an outer brane in $\torX$. By our earlier discussion, its image $\upmu(L)$ under the moment must be a point on a non-compact edge $\ell$. Let us write $q$ for the unique torus fixed point contained in $\ell$ and denote by
\begin{equation*}
	\msf{u},\msf{v},\msf{w} \in \hhh^2_{\gpfont{T}}(\mr{pt},\mbb{Z}) \subset \mr{Lie}(\gpfont{T}_{\mbb{R}})^{\cstar}
\end{equation*}%
\begin{figure}[t]
	\centering
	\begin{tikzpicture}[smooth, baseline={([yshift=-.5ex]current bounding box.center)},scale=0.8]%
		\draw[thick] (0,0) to (1,1);
		\draw[thick] (0,0) to (-1.5,0);
		\draw[thick,dashed] (-1.5,0) to (-2,0);
		\draw[thick] (0,0) to (0,-1);
		\draw[thick,dashed] (0,-1) to (0,-1.5);
		\draw[thick] (1,1) to (2,1);
		\draw[thick,dashed] (2,1) to (2.5,1);
		\draw[thick] (1,1) to (1,2);
		\draw[thick,dashed] (1,2) to (1,2.5);
		\draw[->,very thick,purple] (0,0) to (0.7,0.7);
		\draw[->,very thick,purple] (0,0) to (-0.7,0);
		\draw[->,very thick,purple] (0,0) to (0,-0.7);
		\node at (0,0) {$\bullet$};
		\node at (1,1) {$\bullet$};
		\draw[->,very thick,purple] (-1.3,0) to (-1.3,-0.7);
		\node[color=highlight] at (-1.3,0) {$\bullet$};
		\node[color=highlight,above] at (-1.3,0) {$\upmu(L)$};
		\node[left,purple] at (-1.3,-0.35) {$\mathsf{f}$};
		\node[below,purple] at (-0.45,-0.05) {$\mathsf{u}$};
		\node[right,purple] at (0,-0.4) {$\mathsf{v}$};
		\node[right,purple] at (0.3,0.2) {$\mathsf{w}$};
		\node[left] at (-2,0) {$\ell$};
		\node[above] at (0,0) {$q$};
	\end{tikzpicture}%
	\caption{The image of the toric skeleton (black) of $\Tot \cO_{\bP^1}(-1) \oplus \cO_{\bP^1}(-1)$ with an outer brane $L$ (blue) in framing $f=0$ (red) under the moment map.}
	\label{fig: example outer brane}
\end{figure}
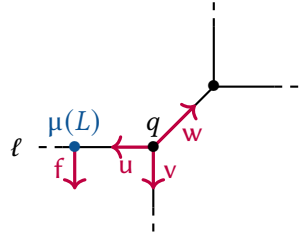%
the weights of the induced $\gpfont{T}$ action on the tangent spaces of three torus orbit closures at this point. Here, $\msf{u} = c^{\gpfont{T}}_1(T_q \ell)$ and $\msf{v}$ is chosen so that $\msf{u} \wedge \msf{v}\geq 0$ as indicated in \Cref{fig: example outer brane}. A \define{framing} of $L$ is the choice of an element $\msf{f}\in \hhh^2_{\gpfont{T}}(\mr{pt},\mbb{Z})$ satisfying\footnote{This definition agrees with the one in \cite[Section~2.5]{FL13:OpenGWtorCY3} and differs from \cite[Definition~6.1]{BBvG2} by a minus sign.}
\begin{equation*}
	\msf{u} \wedge \msf{f} = \msf{u} \wedge \msf{v}\,.
\end{equation*}
Equivalently, we have
\begin{equation}
	\label{eq: defn framing factor}
	\msf{f} = \msf{v} - f \msf{u}
\end{equation}
for some $f\in \mbb{Z}$. Moreover, as an element in $\hhh^2_{\gpfont{T}}(\mr{pt},\mbb{Z})$ one may view $\msf{f}$ as a character $\gpfont{T} \rightarrow \Gm$. This defines a one dimensional subtorus $\gpfont{T}_{\msf{f}}\coloneqq \ker \msf{f}\hookrightarrow \gpfont{T}$ which we will refer to as the \define{framing subtorus}. Finally, let us denote by $\lvert_{\msf{f}=0}$ the restriction map
\begin{equation*}
	\hhh_*^{\gpfont{T}}(\mr{pt},\mbb{Z}) \longrightarrow \hhh_*^{\gpfont{T}_{\msf{f}}}(\mr{pt},\mbb{Z})
\end{equation*}
as this morphism essentially sets the weight $\msf{f}$ equal to zero.

\begin{defn}
	\label{defn: toric triple}
	A \define{toric triple} $(\torX,L,\msf{f})$ consists of
	\begin{itemize}
		\item a smooth semiprojective toric Calabi--Yau threefold $\torX$,
		\item an outer brane $L$,
		\item a framing $\msf{f}$.
	\end{itemize}
\end{defn}
\begin{rmk}
	\label{rmk: not semiproj triple}
	To cover certain interesting cases it will actually be necessary to relax the above definition slightly. Concretely, we would like to include the case where $\torX$ is not semiprojective in our discussion as well. In this case we replace the second point in the above definition with the choice of some non-compact one dimensional torus orbit closure $\ell$.
\end{rmk}

\subsubsection{Open Gromov--Witten invariants} Given a toric triple $(\torX,L,\msf{f})$ and a collection $\mathbf{c}=(c_1,\ldots,c_n)$ of positive integers we denote by
\begin{equation*}
	\Mbar_{g,\mathbf{c},\upbeta^{\prime}}(\torX,L)
\end{equation*}
the moduli space parametrising genus $g$ stable maps $f:(C,\partial C) \rightarrow (\torX,L)$ whose domain $C$ is a Riemann surface with $n$ labelled boundary components $\partial C = \partial C_1 \sqcup \ldots \sqcup \partial C_n$ such that for all $i\in \{1,\ldots,n\}$
\begin{equation*}
	f_\cstar [\partial C_i] = c_i [S^1]\in \hhh_1(L,\mbb{Z}) \quad\text{ and }\quad f_\cstar [C] = \upbeta^{\prime} + \textstyle{\sum_i} c_i [B] \in \hhh_2(\torX,L,\mbb{Z})\,.
\end{equation*}
Here, $\upbeta^{\prime}$ a curve class in $X$ and $B$ is the unique $\gpfont{T}_{\mbb{R}}$-preserved disk stretching out from $q$ to the circle $S^1$ in which $L$ and $\ell$ intersect. The virtual dimension of this moduli problem is zero. We refer to \cite{KL02:StabMapsBoundary} for details.

Moduli spaces parametrising stable maps with boundaries are (if they actually exist) usually notoriously hard to work with. In our case there is however a way out. Since the Calabi--Yau torus $\gpfont{T}_{\mbb{R}}$ leaves $L$ invariant the $\gpfont{T}_{\mbb{R}}$-action on $\torX$ lifts to an action on $\Mbar_{g,\mathbf{c},\upbeta^{\prime}}(\torX,L)$. Then as opposed to its ambient space the (anticipated) fixed locus
\begin{equation*}
	\Mbar_{g,\mathbf{c},\upbeta^{\prime}}(\torX,L)^{\gpfont{T}_{\mbb{R}}}
\end{equation*}
is a compact complex orbifold. So as in \cite{FL13:OpenGWtorCY3} we define open Gromov--Witten invariants via virtual localisation:
\begin{equation}
	\label{eq: defn open GW}
	\OGW{g,\mathbf{c},\upbeta^{\prime}}{\torX,L,\msf{f}}\coloneqq \int_{[\Mbar_{g,\mathbf{c},\upbeta^{\prime}}(\torX,L)^{\gpfont{T}_{\mbb{R}}}]^{\vir}_{{\gpfont{T}_{\mbb{R}}}}} \frac{1}{e^{\gpfont{T}_{\mbb{R}}} (N^{\vir})} ~ \bigg\lvert_{\msf{f}=0} \in \bQ\,.
\end{equation}
The above integral produces a rational function in the two equivariant parameters of homogenous degree zero. Only the restriction to the framing subtorus yields a rational number. This is how definition \eqref{eq: defn open GW} depends on the choice of framing $\msf{f}$.

We recall a result of Fang and Liu which allows us to express open Gromov--Witten invariants of $(\torX,L,\msf{f})$ in terms of descendant invariants of $\torX$.

\begin{thm}
	\label{prop: open GW to psi insertion}
	\cite[Proposition 3.4]{FL13:OpenGWtorCY3} If $(\torX,L,\msf{f})$ is a toric triple then
	\begin{equation}
	\label{eq: open GW to psi insertion}
			\OGW{g,\mathbf{c},\upbeta^{\prime}}{\torX,L,\msf{f}} =
			\prod_{i=1}^n (-1)^{f c_i} \frac{\textstyle{\prod_{k=1}^{c_i-1}(f c_i +k)}}{\msf{u} \cdot c_i!}
			\int_{[\Mbar_{g,n,\upbeta^{\prime}}(\torX)^{\gpfont{T}}]^{\vir}_{\gpfont{T}}} \frac{1}{e^{\gpfont{T}}(N^{\vir})} \, \prod_{i=1}^n \frac{\mr{ev}_i^{\cstar} \upphi}{ \big(\tfrac{\msf{u}}{c_i}-\uppsi_i\big) }  ~\bigg\lvert_{\msf{f}=0}
	\end{equation}
	where $\upphi$ is the $\gpfont{T}$-equivariant Poincar\'{e} dual of the torus fixed point $q$ contained in the one dimensional torus orbit closure $\ell$ intersected by $L$.
\end{thm}
\begin{rmk}
	\label{rmk: open equals formal toric}
	In \cite{LLLZ09:MathTopVert} Li, Liu, Liu and Zhou generalise the definition of open Gromov--Witten invariants to the case when $\torX$ is not semiprojective. (This assumption was necessary for the construction of Aganagic--Vafa Lagrangian submanifolds.) In this case the datum of an outer brane gets replaced by the choice of a non-compact one dimensional torus orbit closure $\ell$ and open Gromov--Witten invariants are defined via formal relative invariants. It is shown in \cite[Corollary 3.3]{FL13:OpenGWtorCY3} that these invariants still satisfy \Cref{prop: open GW to psi insertion}.
\end{rmk}

\subsection{Statement of the main correspondence}
\label{sec: log 3fold to open}
Let $(S \vertspace D_1+D_2)$ be a two-component Looijenga pair and $\upbeta$ an effective curve class in $S$. We demand that this data satisfies \Cref{assumption: star} which we recall means that
\begin{itemize}
	\item $D_i \cdot \upbeta >0$, $i\in\{1,2\}$,
	\item $D_2 \cdot D_2\geq 0$,
	\item $(S\vertspace D_1)$ deforms into a pair $(\defS \vertspace \defDone)$ with $\defS$ a smooth projective toric surface and $\defDone$ a toric hypersurface.
\end{itemize}
\begin{rmk}
	\label{rmk: explanation deformation property}
	By the last condition we mean that there is a logarithmically smooth morphism of fine saturated logarithmic schemes $\mc{S} \rightarrow T$ with $T$ irreducible and integral together with a line bundle $\mc{L}$ on $\mc{S}$. Moreover, there are regular points $t,t^{\prime}: \Speck \rightarrow T$ with fibres
	\begin{equation*}
		\mc{S}_{t} = (S\vertspace D_1)\, , \qquad \mc{S}_{t^{\prime}} = (\defS\vertspace \defDone)
	\end{equation*}
	on which $\mc{L}$ restricts to $\mc{L}_t = \cO_{S}(-D_2)$ and $\mc{L}_{t'} = \upomega_{\defS}(\defDone)$ where $\upomega_{\defS}$ is the canonical bundle on $\defS$.
\end{rmk}

Given $(\defS \vertspace \defDone)$ as above we write $Y\coloneqq \Tot \upomega_{\defS}(\defDone)$ and write $D$ for the preimage of $\defDone$ under the projection $\uppi: Y \rightarrow \defS$. We remark that all compact one dimensional torus orbit closures in $Y$ come from the toric boundary of $\defS$ embedded via the zero section $\defS \hookrightarrow Y$. This way we may especially view $\defDone$ as a toric curve in $Y$.

\begin{construction}
	\label{constr: toric triple}
	To $(S \vertspace D_1+D_2)$ as above we associate the toric triple $(\torX,L,\msf{f})$ where
	\begin{equation*}
		\torX\coloneqq  Y \setminus D
	\end{equation*}
	and $L\subset \torX \subset Y$ is an outer brane intersecting a compact one dimensional torus orbit closure $\ell \subset Y$ adjacent to $\defDone$. The action of the Calabi--Yau torus $\gpfont{T}$ of $X$ naturally extends to $Y$. Writing $F$ for the fibre $\uppi^{-1}(p)$ over the point $p = \ell \cap \defDone$ we set $\msf{f} \coloneqq c_1^{\gpfont{T}} (F)$ to be the weight of the $\gpfont{T}$-action on $F$.
\end{construction}

We are now able to state our main result. For this we fix $c_1,\ldots,c_n > 0$ with $\sum_i c_i = D_1 \cdot \upbeta$ and denote by $\mathbf{\hat{c}}$ the contact datum
\begin{equation}
	\label{eq: defn c hat}
	\mathbf{\hat{c}} = \bigg(\raisebox{-0.9em}{\begin{minipage}{13.2em}$\hspace{0.2em}\begin{matrix}
				0 & \cdots & 0 & c_1 & \cdots & c_n & 0\\
				0 & \mathclap{\underbrace{\makebox[4.5em]{$\cdots$}}_{\text{$m$ times}}} & 0 & 0 & \cdots & 0 & D_2 \cdot \upbeta
			\end{matrix}$\end{minipage}}\bigg)
\end{equation}
along $D_1+D_2$. We assume $m\leq n$.

\begin{thm} \label{thm: log open} (\Cref{thm: intro log-open})
	Under \Cref{assumption: star} we have
	\begin{equation}
		\label{eq:log-open}
		\sum_{g\geq 0} \hbar^{2g-2} ~ \OGW{g,\mathbf{c},\upbeta^{\prime}}{\torX,L,\msf{f}}  = \frac{(-1)^{D_1 \cdot \upbeta}}{m! \, \textstyle{\prod_{i=0}^{m-1}} c_{n-i}} ~ \frac{(-1)^{D_2 \cdot \upbeta +1}}{2 \sin \tfrac{(D_2 \cdot \upbeta)\hbar}{2}} ~ \sum_{g\geq 0} \hbar^{2g-1} ~ \LGW{g,\mathbf{\hat{c}},\upbeta}{S \vertspace D_1+D_2}{(-1)^g \uplambda_g \,\Pi_{i=1}^n \mr{ev}_i(\mr{pt}) }
	\end{equation}
	where $\upbeta^{\prime} = \upiota^{\cstar} \big(\upbeta - (D_1 \cdot \upbeta) [\ell]\big)$ and $\upiota$ is the open inclusion $\torX \hookrightarrow Y$.
\end{thm}

See \Cref{sec:dP3 0 2} for an extended example illustrating \Cref{constr: toric triple} and an application of the above \namecref{thm: log open}.

\begin{rmk}
	Note that \Cref{constr: toric triple} depends on the choice which of the two torus orbit closures adjacent to $D_1$ the Lagrangian submanifold $L$ is intersecting. Hence, a priori, one should expect the right-hand side of \eqref{eq: log 3fold to open} to depend on this choice as well and it should come as a surprise that \Cref{thm: log open} actually demonstrates an independence.
\end{rmk}

\begin{rmk}
	If $\torX$ as constructed in \labelcref{constr: toric triple} turns out to be non-semiprojective the left-hand side of equation \eqref{eq:log-open} can still be defined as described in \Cref{rmk: open equals formal toric}.
\end{rmk}

\begin{rmk}
	Let us quickly comment on the importance of \Cref{assumption: star}. The first two points guarantee that the moduli stack of relative stable maps to $(Y \vertspace D)$ is proper which is required in \Cref{thm: GNS main theorem}. The last point in \Cref{assumption: star} is necessary for $\torX$ to be toric which is in turn required in the definition of open Gromov--Witten invariants we use. The condition is rather technical but allows us treat interesting cases such as the example we will discuss in \Cref{sec:dP3 0 2}.
\end{rmk}

We observe that in the light of \Cref{thm: GNS main theorem} and \Cref{prop: open GW to psi insertion} it suffices to prove the following \namecref{prop: log threefold to descedant} in order to deduce \Cref{thm: log open}.

\begin{prop}
	\label{prop: log threefold to descedant}
	Under \Cref{assumption: star} we have
	\begin{equation}
		\label{eq: log 3fold to open}
	\begin{split}
		&\LGW{g,\mathbf{\tilde{c}},\upbeta}{Y  \vertspace  D}{\Pi_{i=1}^n\mr{ev}_i^{\cstar}(\uppi^{\cstar} \mr{pt})}\\
		&\hspace{3em}= (-1)^{D_1 \cdot \upbeta} m! \left( \textstyle{\prod_{i=0}^{m-1}}c_{n-i}\right) \prod_{i=1}^n (-1)^{f c_i} \frac{\textstyle{\prod_{k=1}^{c_i-1}(f c_i +k)}}{\msf{u} \cdot c_i!} ~ \int_{[\Mbar_{g,n,\upbeta^{\prime}}(\torX)^{\gpfont{T}}]^{\vir}_{\gpfont{T}}} \frac{1}{e^{\gpfont{T}}(N^{\vir})} \, \prod_{i=1}^n \frac{\mr{ev}_i^{\cstar} \upphi}{ \tfrac{\msf{u}}{c_i}-\uppsi_i } ~\bigg\lvert_{\msf{f}=0}
	\end{split}
	\end{equation}
	where $\upbeta^{\prime}=\upiota^{\cstar} (\upbeta - (D_1 \cdot \upbeta) [\ell])$ and $\upphi$ is the $\gpfont{T}$-equivariant Poincar\'{e} dual of the torus fixed point $q$ in $\ell\cap X$.
\end{prop}

We will prove this statement in two steps:
\begin{itemize}
	\item In \Cref{sec: log 3fold to open proof} we first establish the special case $m=0$ (\Cref{thm: log open no interior marking}).

	\item In \Cref{sec: reduction} we prove the general case via a reduction argument (\Cref{thm: comparison point conditions}).
\end{itemize}
But before we embark on proving \Cref{prop: log threefold to descedant}, let us quickly convince ourselves that the \namecref{prop: log threefold to descedant} indeed implies \Cref{thm: log open}.

\begin{proof}[Proof of \Cref{thm: log open}]
	Comparing formula \eqref{eq: open GW to psi insertion} and \eqref{eq: log 3fold to open} we deduce
	\begin{equation*}
		\OGW{g,\mathbf{c},\upbeta^{\prime}}{\torX,L,\msf{f}} = \frac{(-1)^{D_1 \cdot \upbeta}}{m! \, \textstyle{\prod_{i=0}^{m-1}} c_{n-i}} \cdot \LGW{g,\mathbf{\tilde{c}},\upbeta}{Y  \vertspace  D}{\Pi_{i=1}^n\mr{ev}_i^{\cstar}(\uppi^{\cstar} \mr{pt})}\,.
	\end{equation*}
	Now by deformation invariance \cite[Theorem 0.3]{GS13:LogGWInv} (see also \cite[Appendix A]{MR20:DescLogGWTrop}) the Gromov--Witten invariants of $(S \vertspace D_1)$ and $(\defS \vertspace \defDone)$ agree. In combination with \Cref{thm: GNS main theorem} this therefore gives the identity stated in \Cref{thm: log open}.
\end{proof}

\section{Step I: no interior markings}
\label{sec: log 3fold to open proof}
This section is devoted to the proof of a special instance of \Cref{prop: log threefold to descedant}.

\begin{prop}
	\label{thm: log open no interior marking}
	The statement of \Cref{prop: log threefold to descedant} holds if there are no interior markings ($m=0$).
\end{prop}

We prove the \namecref{thm: log open no interior marking} via virtual localisation \cite{GP97:virtloc} which allows us to decompose the Gromov--Witten invariant on the left-hand side of equation \eqref{eq: log 3fold to open} into contributions labelled by the fixed loci of a torus action on the moduli stack of relative stable maps to $(Y\vertspace D)$. These loci split into two classes: The ones for which the target $(Y\vertspace D)$ is generically expanded or unexpanded. A careful analysis of the latter shows that their contribution to the overall Gromov--Witten count is precisely the expression we find on the right-hand side of equation \eqref{eq: log 3fold to open} (\Cref{prop: contrib simple locus}). It therefore remains to show that the contribution of all other fixed loci vanishes. This last step is carried out in \Cref{sec: composite locus}.

We start with a careful analysis of the target geometry $(Y\vertspace D)$ where we adopt the notation introduced in \Cref{sec: log 3fold to open}. To unload notation throughout this section we will assume that $S$ is already toric and $D_1$ is a toric hypersurface which means we can take $(\defS\vertspace \defDone) = (S\vertspace D_1)$.

\subsection{The geometric setup}
\label{sec: open geometric setup}
The Calabi--Yau torus $\gpfont{T}$ of $\torX$ also naturally acts on $Y$ as both varieties share the same dense open torus. The compact one dimensional torus orbit closures of $Y$ are given by the irreducible components of the toric boundary of $S$ embedded into $Y$ by the zero section. We write $\ell\hookrightarrow Y$ for the component intersected by $L$. By construction $\ell$ intersects $D_1\hookrightarrow Y$ transversely in a torus fixed point which we denote $p$. We write $q$ for the second torus fixed point of $\ell$.

\begin{figure}%
	\centering%
	\begin{tikzpicture}[smooth, baseline={([yshift=-.5ex]current bounding box.center)}]%
		\draw[thick] (-1.5,0) to (4.5,0);
		\draw[thick,dashed] (-1.5,0) to (-2,0);
		\draw[thick,dashed] (4.5,0) to (5,0);

		\draw[thick] (0,0) to (-1,-2);

		\draw[thick] (-1,-2) to (0,-3);

		\draw[thick] (-1,-2) to (-2,-2.25);
		\draw[thick,dashed] (-2,-2.25) to (-2.4,-2.35);

		\draw[thick] (3,0) to (3.3,-1.5);
		\draw[thick,dotted] (3.3,-1.5) to (3.4,-2) to (2.4,-3.2) to (0.5,-3.5) to (0,-3);
		\draw[thick,dotted] (3.4,-2) to (3.9,-2.2);
		\draw[thick,dotted] (2.4,-3.2) to (2.7,-3.5);
		\draw[thick,dotted] (0.5,-3.5) to (0.4,-3.9);
		\node[below] at (1.5,0) {$D_1$};
		\node[below] at (-1.5,0) {$F$};
		\node[right] at (-0.55,-1.25) {$\ell$};
		%
		\node at (0,0) {$\bullet$};
		\node[below] at (0.15,0) {$p$};
		\node at (3,0) {$\bullet$};
		%
		\node at (-1,-2) {$\bullet$};
		\node[below] at (-1.03,-2.03) {$q$};
		\node[above,purple] at (-0.4,-0.05) {$\msf{f}$};
		\node[above,purple] at (0.4,-0.05) {$-\msf{f}$};
		\node[above,purple] at (2.6,-0.05) {$\msf{f}$};
		\node[above,purple] at (3.4,-0.05) {$-\msf{f}$};
		\node[left,purple] at (-0.13,-0.3) {$-\msf{u}$};
		\node[left,purple] at (-0.8,-1.65) {$\msf{u}$};
		\node[left,purple] at (-1.35,-1.95) {$\msf{f}+f\msf{u}$};
		\node[right,purple] at (-0.83,-2.2) {$-\msf{f}-(f+1)\msf{u}$};
	\end{tikzpicture}%
	\caption{The image of the toric skeleton of $Y$ under the moment map.}
	\label{fig: toric diagram}
\end{figure}
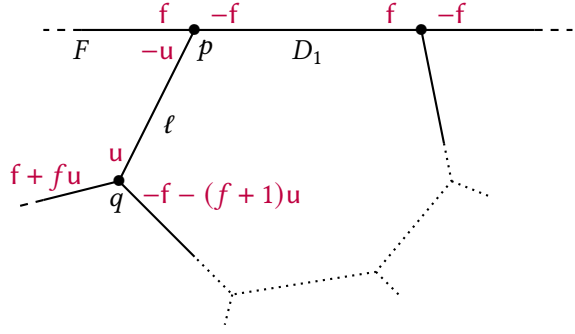%

\Cref{fig: toric diagram} shows the image of one dimensional torus orbit closures of $Y$ under the moment map locally around $D_1$. In this sketch dashed lines correspond to non-compact orbit closures while the dotted lines indicate that the toric boundary of $S\hookrightarrow Y$ forms a circle. Moreover, the red labels represent the weights with which the torus $\gpfont{T}$ is acting on the tangent spaces of the respective torus orbit closures at a fixed point. The observation (see \cite[Lemma 2.7]{GNS23:bicyclic}) that
\begin{equation}
	\label{eq: T weight on D1}
	c_1^{\gpfont{T}}(T_{p}F) = - c_1^{\gpfont{T}}(T_{p}D_1) = \msf{f}
\end{equation}
will be crucial in our localisation calculation later. Moreover, we also note that the framing factor $f$ introduced in \eqref{eq: defn framing factor} can be identified with the degrees
\begin{equation*}
	f = \deg \cO_{S}(-D_2)\big\lvert_{\ell}=  -\deg N_{\ell}Y - 1\,.
\end{equation*}
The toric diagram of $\torX = Y \setminus D$ is obtained from the one displayed in \Cref{fig: toric diagram} by erasing the top vertical line. In this process the divisor $\ell$ gets decompactified and so $L$ is indeed an outer brane.

\subsection{Initialising the localisation}
\Cref{constr: toric triple} involved the choice of a stratum $\ell$ the outer brane is intersecting. We now do an analogous choice by supporting the cycle whose pushforward gives the Gromov--Witten invariant $\LGW{g,\mathbf{c},\upbeta}{Y  \vertspace  D}{\Pi_{i=1}^n\mr{ev}_i^{\cstar}(\uppi^{\cstar} \mr{pt})}$ on a substack with evaluations constrained to the fibre $F=\uppi^{-1}(p)$ where $\uppi$ is the projection $Y\rightarrow S$:
\begin{equation}
	\label{eq: confined evaluation}
	\begin{tikzcd}
		\Mbar_{g,\mathbf{c},\upbeta}(Y \vertspace D)\lvert_{F} \arrow[r, hook] \arrow[d] \ar[rd,phantom,"\square",start anchor=center,end anchor=center] & \Mbar_{g,\mathbf{c},\upbeta}(Y \vertspace D) \arrow[d]\\
		F^n \arrow[r,hook] & D^n\,.
	\end{tikzcd}
\end{equation}
Pulling back the obstruction theory along the horizontal arrow we obtain
\begin{equation*}
	\LGW{g,\mathbf{c},\upbeta}{Y  \vertspace  D}{\Pi_{i=1}^n\mr{ev}_i^{\cstar}(\uppi^{\cstar} \mr{pt})} = \deg [\Mbar_{g,\mathbf{c},\upbeta}(Y \vertspace D)\lvert_{F}]^{\vir}\,.
\end{equation*}
This choice will turn out rather useful when we will calculate the invariant via virtual localisation. Regarding this we remark that the $\gpfont{T}$-action on $Y$ lifts to an action on $\Mbar_{g,\mathbf{c},\upbeta}(Y \vertspace D)\lvert_{F}$ since $D$ and $F$ are both preserved by $\gpfont{T}$. Especially, this also provides us with an action of the framing subtorus $\gpfont{T}_{\msf{f}}$ on the moduli stack.

Now equation \eqref{eq: T weight on D1} tells us that the divisor $D$ is pointwise fixed under $\gpfont{T}_{\msf{f}}$ since the restriction $\gpfont{T}_{\msf{f}}\hookrightarrow \gpfont{T}$ effectively sets the weight $\msf{f}$ equal to zero. This means we are in a situation where the analysis of Graber and Vakil \cite{GV05:RelVirtLoc} of virtual localisation in the context of stable maps relative a smooth divisor applies. Since their analysis was carried out in the setting of relative stable maps to expanded degenerations as introduced by Li \cite{Li01:RelStabMaps,Li02:Degen} we quietly pass to this modification of $\Mbar_{g,\mathbf{c},\upbeta}(Y \vertspace D)$ until the end of this section without actually changing our notation. Since eventually we push all cycles forward to a point it does not matter which birational model we choose for our moduli stack \cite{AMW14:Comparison}. See also \cite{MR19:LocalizeLog} for an analysis of torus localisation in the context of Kim's logarithmic stable maps to expanded degenerations \cite{Kim10:LogStabMaps}.

Following Graber and Vakil \cite{GV05:RelVirtLoc} we split the $\gpfont{T}_{\msf{f}}$-fixed locus $\Mbar_{g,\mathbf{c},\upbeta}(Y \vertspace D)^{\gpfont{T}_{\msf{f}}}$ into two distinguished (not necessarily connected or irreducible) components. First, there is the simple fixed locus
\begin{equation*}
	\Mbar_{g,\mathbf{c},\upbeta}(Y \vertspace D)^{\simple}
\end{equation*}
whose general points are stable maps with target $Y$. The complement of this is called composite fixed locus and consequently its general points correspond to maps to non-trivial expanded degenerations of $(Y \vertspace D)$. We will denote this component by
\begin{equation*}
	\Mbar_{g,\mathbf{c},\upbeta}(Y \vertspace D)^{\mathrm{com}}\,.
\end{equation*}
As we did in \eqref{eq: confined evaluation} we form fibre products
\begin{equation*}
	\Mbar_{g,\mathbf{c},\upbeta}(Y \vertspace D)\lvert_{F}^{\simple} ~\coloneqq~ \Mbar_{g,\mathbf{c},\upbeta}(Y \vertspace D)^{\simple} \times_{D^n} F^n\,,\hspace{3em} \Mbar_{g,\mathbf{c},\upbeta}(Y \vertspace D)\lvert_{F}^{\mathrm{com}}~\coloneqq~\Mbar_{g,\mathbf{c},\upbeta}(Y \vertspace D)^{\mathrm{com}}\times_{D^n} F^n \,.
\end{equation*}
Then virtual localisation \cite{GP97:virtloc} gives
\begin{equation}
	\label{eq: outcome localisation intermed}
	\begin{split}
		\LGW{g,\mathbf{c},\upbeta}{Y  \vertspace  D}{\Pi_{i=1}^n\mr{ev}_i^{\cstar}(\uppi^{\cstar} \mr{pt})} = \,&\deg_{\gpfont{T}_{\msf{f}}} \frac{1}{e^{\gpfont{T}_{\msf{f}}}(N^{\vir}_{\simple})}\cap [\Mbar_{g,\mathbf{c},\upbeta}(Y \vertspace D)\lvert_{F}^{\simple}]^{\vir}\\
		+ \,&\deg_{\gpfont{T}_{\msf{f}}} \frac{1}{e^{\gpfont{T}_{\msf{f}}}(N^{\vir}_{\mathrm{com}})}\cap [\Mbar_{g,\mathbf{c},\upbeta}(Y \vertspace D)\lvert_{F}^{\mathrm{com}}]^{\vir}\,.
	\end{split}
\end{equation}
We will inspect the contribution of the simple and composite fixed locus to the Gromov--Witten invariant individually in the next two subsections.

\subsection{The simple fixed locus}
Closely following \cite{GV05:RelVirtLoc}, we start with an analysis of the simple fixed locus. By definition the image of a relative stable map $f:C\rightarrow Y$ parametrised by
\begin{equation*}
	\Mbar_{g,\mathbf{c},\upbeta}(Y \vertspace D)\lvert_{F}^{\simple}
\end{equation*}
intersects $D$ transversely, ie.\ $f^{\cstar} D = \sum_{i=1}^n c_i x_i$ as Cartier divisors where $x_1,\ldots,x_n$ denote the markings. If we write
\begin{equation*}
	C_1,\ldots,C_n
\end{equation*}
for the irreducible components containing these markings then since $f:C\rightarrow Y$ lies in the $\gpfont{T}_{\msf{f}}$-fixed locus and by assumption none of the components $C_i$ get contracted there must be a (non-trivial) lift of the action $\gpfont{T}_{\msf{f}}$ on $C_i$ (possibly after passing to some cover of $\gpfont{T}_{\msf{f}}$) making $f\lvert_{C_i}$ $\gpfont{T}_{\msf{f}}$-equivariant. As a consequence, for all $i\in\{1,\ldots,n\}$ both $C_i$ and its image under $f$ must be rational curves. Moreover, writing $q_i$ for the second $\gpfont{T}_{\msf{f}}$-fixed point of the image rational curve, $f\lvert_{C_i}$ is the unique degree $c_i$ cover maximally ramified over $f(x_i)$ and $q_i$.

Now since we confined the evaluation morphisms to the fibre $F\subset Y$, we deduce that the image of each component $C_i$ is contained in $\uppi^{-1}(\ell)$ where $\uppi$ is the projection $Y\rightarrow S$ and $\ell$ is as in \Cref{fig: toric diagram}. However, since the pullback of the line bundle $\cO_{S}(-D_2)$ under $\uppi \circ f$ is negative we deduce that $f$ has to factor through the zero section ${S}\hookrightarrow Y$. Hence, the image of each irreducible component $C_i$ has to be the curve $\ell$ already. This also means that $q_i=q$ for all $i\in\{1,\ldots,n\}$.

Now if we write $C' \sqcup \bigsqcup_{i=1}^n C_i$ for the connected components of the partial normalisation of $C$ along the nodes $y_i\in C_i$ mapping to $q$ we observe that $f\lvert_{C'}$ is a genus $g$, class
\begin{equation*}
	\upbeta^{\prime} = \upiota^{\cstar} \big(\upbeta - (D_1 \cdot \upbeta) [\ell]\big)
\end{equation*}
$\gpfont{T}_{\msf{f}}$-fixed stable map to $\torX = Y \setminus D$ with markings $y_1,\ldots,y_n$ confined to $q$. Hence, forming the fibre square
\begin{equation*}
\begin{tikzcd}
	\Mbar_{g,n,\upbeta^{\prime}}(\torX) \lvert_{q} \arrow[r,hook] \arrow[d]\ar[rd,phantom,"\square",start anchor=center,end anchor=center] & \Mbar_{g,n,\upbeta^{\prime}}(\torX) \arrow[d,"\Pi_i\mr{ev}_{y_i}"] \\
	q \arrow[r,hook] & \torX^n
\end{tikzcd}
\end{equation*}
we can summarise the discussion up to now as follows.

\begin{lem}
	We have
	\begin{equation}
		\label{eq: simple fixed locus}
		\Mbar_{g,\mathbf{c},\upbeta}(Y \vertspace D)\lvert_{F}^{\simple} = \Big[\, \Mbar_{g,n,\upbeta^{\prime}}(\torX) \lvert_{q}^{\gpfont{T}_{\msf{f}}} \,\Big/\, \textstyle{\prod_{i=1}^{n}}\,\mbb{Z}/ c_i \mbb{Z} \,\Big]\,.
	\end{equation}
\end{lem}

Using this description of the simple fixed locus we are now able to determine its contribution to the Gromov--Witten invariant.

\begin{prop}
	\label{prop: contrib simple locus}
	We have
	\begin{equation}
	\label{eq: contrib simple locus}
	\begin{split}
		\deg_{\gpfont{T}_{\msf{f}}} \frac{1}{e^{\gpfont{T}_{\msf{f}}}(N^{\vir}_{\simple})}\cap [\Mbar_{g,\mathbf{c},\upbeta}(Y \vertspace D)\lvert_{F}^{\simple}]^{\vir}_{\gpfont{T}_{\msf{f}}} & = (-1)^{D_1 \cdot \upbeta} \prod_{i=1}^n (-1)^{f c_i} \frac{\textstyle{\prod_{k=1}^{c_i-1}(f c_i +k)}}{\msf{u} \cdot c_i!} \\
		& \hspace{3em} \times \int_{[\Mbar_{g,n,\upbeta^{\prime}}(\torX)^{\gpfont{T}_{\msf{f}}}]^{\vir}_{\gpfont{T}_{\msf{f}}}} \frac{1}{e^{\gpfont{T}_{\msf{f}}}(N^{\vir})} \, \prod_{i=1}^n \frac{\mr{ev}_i^{\cstar} \upphi}{\tfrac{\msf{u}}{c_i}-\uppsi_i}
	\end{split}
	\end{equation}
	where $\upphi$ is the $\gpfont{T}_{\msf{f}}$-equivariant Poincar\'{e} dual of $q$.
\end{prop}

\begin{proof}
	With our characterisation \eqref{eq: simple fixed locus} of the simple fixed locus it remains to analyse the difference between $N^{\vir}_{\simple}$ and the virtual normal bundle $N^{\vir}$ of the fixed locus
	\begin{equation*}
		\Mbar_{g,n,\upbeta^{\prime}}(\torX)^{\gpfont{T}_{\msf{f}}} \hookrightarrow \Mbar_{g,n,\upbeta^{\prime}}(\torX)\,.
	\end{equation*}
	First of all note that we get an overall factor
	\begin{equation*}
		\prod_{i=1}^n \frac{1}{c_i}
	\end{equation*}
	from the quotient. Moreover, partially normalising along the nodes $y_1,\ldots,y_n$ connecting $C'$ with $C_1,\ldots,C_n$ we see that each degree $c_i$ cover
	\begin{equation*}
		f_i \coloneqq f\lvert_{C_i} ~~ :~~ C_i \rightarrow \ell
	\end{equation*}
	contributes with factors
	\begin{align*}
		\frac{1}{e^{\gpfont{T}_{\msf{f}}}\Big(\hhh^\bullet \big(C_i , f_i^{\cstar} T^{\mr{log}}_{\ell|p} \big)^{\mr{mov}}\Big)} & = \frac{1}{c_i !} \bigg(\frac{\msf{u}}{c_i}\bigg)^{-c_i} \\
		\frac{1}{e^{\gpfont{T}_{\msf{f}}}\Big(\hhh^\bullet \big(C_i , f_i^{\cstar} N_{\ell}{S} \big)^{\mr{mov}}\Big)} & = \frac{1}{(-(f+1)c_i)!}  \bigg(\frac{\msf{u}}{c_i}\bigg)^{(f+1) c_i}\\
		\frac{1}{e^{\gpfont{T}_{\mathsf{f}}} \Big( \hhh^\bullet \big(C_i, f_i^{\cstar}  \mathcal{O}_{{S}}(-D_2) \big)^{\mr{mov}}\Big)} & = (-1)^{fc_i + 1} \, (-fc_i-1)! \bigg(\frac{\msf{u}}{c_i}\bigg)^{-f c_i-1}
	\end{align*}
	assuming that $f = \deg \cO_{{S}}(-D_2)\lvert_{\ell} < 0$. Note that if we multiply all these factors together we indeed obtain the top line on the right-hand side of \eqref{eq: contrib simple locus}. The remaining case $f = 0$ can be treated similarly. This is the only case left to check since $f\leq 0$ due to our assumption that $D_2$ is smooth and $D_2\cdot D_2 \geq 0$. Indeed, from these assumptions it follows that $\ell$ intersects $D_2$ non-negatively (as does any smooth curve in $S$) and so $f = - D_2 \cdot \ell \leq 0$.
	
	Regarding the second line in \eqref{eq: contrib simple locus} let us only remark that the insertions
	\begin{equation*}
		\tfrac{\msf{u}}{c_i}-\uppsi_i
	\end{equation*}
	for $i\in\{1,\ldots,n\}$ come from smoothings of the nodes $y_1,\ldots,y_n$.
\end{proof}

\subsection{Vanishing of composite contributions}
\label{sec: composite locus}
We observe that the integral on the right-hand side of equation \eqref{eq: log 3fold to open} is the result of computing the one in \eqref{eq: contrib simple locus} via $\gpfont{T}$-virtual localisation. Hence, to prove \Cref{thm: log open no interior marking} it suffices to show

\begin{prop}
	\label{prop: composite contrib vanishes}
	We have
	\begin{equation*}
		\deg_{\gpfont{T}_{\msf{f}}} \frac{1}{e^{\gpfont{T}_{\msf{f}}}(N^{\vir}_{\mathrm{com}})}\cap [\Mbar_{g,\mathbf{c},\upbeta}(Y \vertspace D)\lvert_{F}^{\mathrm{com}}]^{\vir} = 0\,.
	\end{equation*}
\end{prop}
The proof of this \namecref{prop: composite contrib vanishes} necessitates a careful analysis of the composite fixed locus for which we will closely follow \cite[Section 3]{GV05:RelVirtLoc}.

\subsubsection{The setup}
We first set the notation required in this section. Denote by
\begin{equation}
	\label{eq: proj compl Z}
	Z \coloneqq \bP_{D} \big( \cO_D \oplus N_{D}Y\big)
\end{equation}
the projective completion of $N_{D}Y$ and write $D_0$ and $D_\infty$ for its zero and infinity section. Note that $Z$ is a line bundle over
\begin{equation}
	\label{eq: proj compl P}
	P \coloneqq \bP_{D_1} \big( \cO_{D_1} \oplus N_{D_1}{S}\big)\,.
\end{equation}
More precisely, we have $Z\cong \cO_{P}(-2 f)$ where $f$ is a fibre of $P \rightarrow D_1$.

By definition, the composite fixed locus parametrises relative stable maps whose target is an expanded degeneration of $(Y \vertspace D)$ which is the result of gluing $(Y \vertspace D)$ and multiple copies of $(Z \vertspace D_0 + D_\infty)$ to form an accordion. The composite locus decomposes into a disjoint union of components labelled by certain decorated graphs indicating how a relative stable map distributes over the expanded target.

\begin{defn} \label{defn: splitting type}
	A \define{splitting type} $\Gamma$ is the data of
	\begin{enumerate}
		\item a bipartite graph with vertices $V(\Gamma)$, edges $E(\Gamma)$ and legs $L(\Gamma)$;

		\item an assignment of a curve class $\upbeta_v$ to every vertex $v\in V(\Gamma)$;

		\item a genus label $g_v\in \mbb{Z}_{\geq 0}$ for every vertex $v\in V(\Gamma)$;

		\item a contact order $d_e \in \mbb{Z}_{> 0}$ assigned to every edge $e\in E(\Gamma)$;

		\item a labelling $\{1,\ldots,n\} \cong L(\Gamma)$.
	\end{enumerate}
	This data has to be compatible with $(g,\mathbf{c},\upbeta)$ in the usual way. Vertices are separated into $Y$- and $Z$-vertices and we will write $\Gamma_Y$ and $\Gamma_Z$ for the type of relative stable maps (with possibly disconnected domain) to $(Y \vertspace D)$, respectively $(Z \vertspace D_0+D_\infty)$, obtained by cutting $\Gamma$ along its edges. All legs are adjacent to $Z$-vertices.
\end{defn}

With this notation the composite locus can be written as a union of disconnected components labelled by splitting types:
\begin{equation*}
	\Mbar_{g,\mathbf{c},\upbeta}(Y \vertspace D)^{\mathrm{com}} = \bigsqcup_{\Gamma} \, \Fbar_{\Gamma}\,.
\end{equation*}
To characterise the component labelled by a fixed splitting type $\Gamma$, let us denote by
\begin{equation*}
	\Mbar_{\Gamma_Y}(Y \vertspace D)^{\simple}
\end{equation*}
the simple $\gpfont{T}_{\msf{f}}$-fixed locus of relative stable maps of type $\Gamma_Y$ and write
\begin{equation*}
	\Mbar_{\Gamma_Z}(Z \vertspace D_0+D_\infty)^{\rubber}
\end{equation*}
for the moduli stack of type $\Gamma_Z$ relative stable maps to non-rigid $(Z \vertspace D_0+D_\infty)$ and form a substack with constrained evaluations at marking legs
\begin{equation*}
	\begin{tikzcd}
		\Mbar_{\Gamma_Z}(Z \vertspace D_0+D_\infty)\lvert_{F}^{\rubber} \arrow[r, hook] \arrow[d] \ar[rd,phantom,"\square",start anchor=center,end anchor=center] & \Mbar_{\Gamma_Z}(Z \vertspace D_0+D_\infty)^{\rubber} \arrow[d, "\Pi_i \mr{ev}_{i}"]\\
		F^n \arrow[r, hook] & D_\infty^n\,.
	\end{tikzcd}
\end{equation*}
The two moduli stacks support evaluation morphisms to $D \cong D_0$ associated to the edges of $\Gamma$. Let
$\Delta$ denote the diagonal $D^{E(\Gamma)} \hookrightarrow D^{E(\Gamma)}\times D^{E(\Gamma)}$ and form the fibre product
\begin{equation}
	\label{eq: defn composite locus}
	\begin{tikzcd}
		\Mbar_{\Gamma} \arrow[r, hook] \arrow[d] \ar[rd,phantom,"\square",start anchor=center,end anchor=center] & \Mbar_{\Gamma_Y}(Y \vertspace D)^{\simple} \times \Mbar_{\Gamma_Z}(Z \vertspace D_0+D_\infty)\lvert_{F}^{\rubber} \arrow[d]\\
		D^{E(\Gamma)} \arrow[r, hook, "\Delta"] & D^{E(\Gamma)}\times D^{E(\Gamma)} \,.
	\end{tikzcd}
\end{equation}
Then the connected component $\Fbar_{\Gamma}$ of the composite fixed locus is the quotient of $\Mbar_{\Gamma}$ by $\Aut(d_e)$ and moreover, by \cite[Lemma 3.2]{GV05:RelVirtLoc}, the Gysin pullback
\begin{equation*}
	[\Mbar_{\Gamma}]^{\vir} \coloneqq \Delta^! \big( [\Mbar_{\Gamma_Y}(Y \vertspace D)^{\simple}]^{\vir} \times [\Mbar_{\Gamma_Z}(Z \vertspace D_0+D_\infty)\lvert_{F}^{\rubber}]^{\vir} \big)
\end{equation*}
pushes forward to $[\Fbar_{\Gamma}]^{\vir}$ under the quotient morphism. This means we get
\begin{equation}
	\label{eq: comp locus contrib rewritten}
	\deg_{\gpfont{T}_{\msf{f}}} \frac{1}{e^{\gpfont{T}_{\msf{f}}}(N^{\vir}_{\mathrm{com}})}\cap [\Mbar_{g,\mathbf{c},\upbeta}(Y \vertspace D)\lvert_{F}^{\mathrm{com}}]^{\vir}_{\gpfont{T}_{\msf{f}}} = \sum_{\Gamma} \frac{1}{|\Aut(d_e)|}~ \deg_{\gpfont{T}_{\msf{f}}} \frac{1}{e^{\gpfont{T}_{\msf{f}}}(N^{\vir}_{\Gamma})}\cap [\Mbar_{\Gamma}]^{\vir}\,.
\end{equation}

\subsubsection{First vanishing: restricted curve classes}
\label{sec: composite 1st vanishing}
To prove \Cref{prop: composite contrib vanishes} we will show that each term in the sum \eqref{eq: comp locus contrib rewritten} vanishes in several steps. We first recall a lemma from \cite{GNS23:bicyclic}.
\begin{lem}
	\label{lem: GNS23 lemma}
	\cite[Section 2.3.3]{GNS23:bicyclic}
	If there is a $Z$-vertex $v\in V(\Gamma)$ such that $\upbeta_v$ is not a multiple of a fibre of $P\rightarrow D_1$, then
	\begin{equation}
		\label{eq: single term comp locus}
		\deg_{\gpfont{T}_{\msf{f}}} \frac{1}{e^{\gpfont{T}_{\msf{f}}}(N^{\vir}_{\Gamma})}\cap [\Mbar_{\Gamma}]^{\vir} = 0\,.
	\end{equation}
\end{lem}

\begin{rmk}
	The above lemma was proven in \cite{GNS23:bicyclic} under the assumption that $D_1\cdot \upbeta = 0$ but it is straight forward to generalise the proof to the case at hand where $D_1\cdot \upbeta >0$. The key observation is that after localising the left-hand side of \eqref{eq: single term comp locus} with respect to $\gpfont{T}$ one can identify a trivial factor in the obstruction bundle guaranteeing the vanishing.
\end{rmk}

\subsubsection{Second vanishing: star shaped graphs}
\label{sec: composite 2nd vanishing}
From now on we may assume that every $Z$-vertex $v$ is decorated with curve class $\upbeta_v$ which is a multiple of a fibre of $P\rightarrow D_1$.

The following results are inspired by vanishing arguments appearing in the localisation calculations \cite{LY22:OpenClosed,LLLZ09:MathTopVert,LLZ03:MarinoVafaFormula}. We will present more geometric versions of these arguments here.

\begin{lem}
	\label{prop: composite unique Dinfty leg}
	Unless every vertex in $\Gamma_Z$ carries exactly one leg we have
	\begin{equation*}
		[\Mbar_{\Gamma_Z}(Z \vertspace D_0+D_\infty)\lvert_{F}^{\rubber}]^{\vir} = 0\,.
	\end{equation*}
\end{lem}

\begin{proof}
	Clearly, every $Z$-vertex $v$ has to have at least one leg for the moduli stack to be non-empty since $D_\infty \cdot \upbeta_v>0$. So suppose there is some vertex with at least two legs. Since the curve class carried by this vertex is a fibre class and the restriction of the line bundle $Z\rightarrow P$ to a fibre trivialises, the product of the evaluation morphisms of these two markings factors through the diagonal $D_\infty \rightarrow D_\infty^2$. Hence, we have the following commuting diagram with Cartesian squares:
	\begin{equation*}
	\begin{tikzcd}
		\Mbar_{\Gamma_Z}(Z \vertspace D_0+D_\infty)\lvert_{F}^{\rubber} \arrow[r,hook] \arrow[d,"g"] \ar[rd,phantom,"\square",start anchor=center,end anchor=center]& \Mbar_{\Gamma_Z}(Z \vertspace D_0+D_\infty)^{\rubber} \arrow[d] \\
		F^{n-1} \arrow[r,hook,"\overline{\upeta}"] \arrow[d, "f"] \ar[rd,phantom,"\square",start anchor=center,end anchor=center] & D_\infty^{n-1} \arrow[d] \\
		F^{n} \arrow[r,hook,"\upeta"] & D_\infty^{n}\,.
	\end{tikzcd}
	\end{equation*}
	Note that the excess intersection bundle
	\begin{equation*}
		E = f^{\cstar} N_{F^{n}} D_\infty^{n} \big/ N_{F^{n-1}} D_\infty^{n-1}
	\end{equation*}
	is trivial as is every vector bundle on $F^{n-1} \cong \Aaff{n-1}$. Moreover, note that since $\gpfont{T}_{\msf{f}}$ is acting trivially on $D_\infty$ and $F$, this bundle is also $\gpfont{T}_{\msf{f}}$-equivariantly trivial. Therefore, by the excess intersection formula \cite[Theorem 6.3]{Fulton} we have
	\begin{equation*}
		[\Mbar_{\Gamma_Z}(Z \vertspace D_0+D_\infty)\lvert_{F}^{\rubber}]^{\vir} = \upeta^! [\Mbar_{\Gamma_Z}(Z \vertspace D_0+D_\infty)^{\rubber}]^{\vir} = c_1^{\gpfont{T}_{\msf{f}}} (g^{\cstar} E) \cap \overline{\upeta}^![\Mbar_{\Gamma_Z}(Z \vertspace D_0+D_\infty)^{\rubber}]^{\vir} = 0\,.\qedhere
	\end{equation*}
\end{proof}

As another consequence of our assumption that all curve classes of $Z$-vertices are multiples of fibre classes the evaluation morphism associated to edges
\begin{equation*}
	\Mbar_{\Gamma_Z}(Z \vertspace D_0+D_\infty)\lvert_{F}^{\rubber} \longrightarrow D_0^{E(\Gamma)}
\end{equation*}
factors through $F^{E(\Gamma)} \hookrightarrow D_0^{E(\Gamma)}$. If we combine this with the fact that the evaluation morphism
\begin{equation*}
	\Mbar_{\Gamma_Y}(Y \vertspace D)^{\simple} \rightarrow D^{E(\Gamma)}
\end{equation*}
factors through $D^{E(\Gamma)}_1 \hookrightarrow D^{E(\Gamma)}$ we may restrict the evaluations to the intersection point $p = F \cap D_1$. We get the following commuting diagram with Cartesian squares:
\begin{equation}
\label{eq: big commuting diagram}
\begin{tikzcd}
	\Mbar_{\Gamma_Y}(Y \vertspace D)\lvert_{F}^{\simple}\,\times\, \Mbar_{\Gamma_Z}(\bP^1 \vertspace 0+\infty)^{\rubber} \arrow[d] \arrow[r] \ar[rd,phantom,"\square",start anchor=center,end anchor=center] & \Mbar_{\Gamma_Y}(Y \vertspace D)^{\simple} \times \Mbar_{\Gamma_Z}(Z \vertspace D_0+D_\infty)\lvert_{F}^{\rubber} \arrow[d] \\
	p \arrow[d] \arrow[r] \ar[rd,phantom,"\square",start anchor=center,end anchor=center] & (D_1)^{E(\Gamma)} \times F^{E(\Gamma)} \arrow[d] \\
	D^{E(\Gamma)} \arrow[r, "\quad~~\Delta"] & D^{E(\Gamma)} \times D^{E(\Gamma)} \,.
\end{tikzcd}
\end{equation}
Especially, note that we have a fibre square
\begin{equation}
\label{eq: defn xi inclusion}
\begin{tikzcd}
	\Mbar_{\Gamma_Z}(\bP^1 \vertspace 0+\infty)^{\rubber} \arrow[d] \arrow[r,hook] \ar[rd,phantom,"\square",start anchor=center,end anchor=center] & \Mbar_{\Gamma_Z}(Z \vertspace D_0+D_\infty)\lvert_{F}^{\rubber} \arrow[d] \\
	p \arrow[r,hook,"\xi"]  & F^{E(\Gamma)} \,.
\end{tikzcd}
\end{equation}
First, we record the following vanishing.

\begin{lem}
	\label{prop: composite unique D0 leg}
	We have $[\Mbar_{\Gamma}]^{\vir} = 0$ unless every $Z$-vertex has a single adjacent edge.
\end{lem}

\begin{proof}
	The \namecref{prop: composite unique D0 leg} can be proven with an argument similar to the one in the proof of \Cref{prop: composite unique Dinfty leg} using diagram \eqref{eq: big commuting diagram}. We leave this to the reader.
\end{proof}

Moreover, comparing our definition of $\Mbar_{\Gamma}$ in \eqref{eq: defn composite locus} with diagram \eqref{eq: big commuting diagram} we learn that
\begin{equation*}
	\Mbar_{\Gamma} \cong \Mbar_{\Gamma_Y}(Y \vertspace D)\lvert_{F}^{\simple}\,\times\, \Mbar_{\Gamma_Z}(\bP^1 \vertspace 0+\infty)^{\rubber} \,.
\end{equation*}
On the level of virtual fundamental classes we find the following identity.

\begin{lem}
	\label{prop: higher genus vanishing}
	If there is a $Z$-vertex $v$ with $g_v>0$ the cycle $[\Mbar_{\Gamma}]^{\vir}$ vanishes. Otherwise, the cycle equates to
	\begin{equation*}
		[\Mbar_{\Gamma_Y}(Y \vertspace D)^{\simple}\lvert_{F}]^{\vir} \times [\Mbar_{\Gamma_Z}(\bP^1 \vertspace 0+\infty)^{\rubber}]^{\vir}\,.
	\end{equation*}
\end{lem}

\begin{proof}
	Let $\xi$ be as in \eqref{eq: defn xi inclusion}. Then since all squares in \eqref{eq: big commuting diagram} are Cartesian we have
	\begin{align*}
		[\Mbar_{\Gamma}]^{\vir} & = \Delta^! \big( [\Mbar_{\Gamma_Y}(Y \vertspace D)^{\simple}]^{\vir} \times [\Mbar_{\Gamma_Z}(Z \vertspace D_0+D_\infty)\lvert_{F}^{\rubber}]^{\vir}\big) \\[0.3em]
		& = [\Mbar_{\Gamma_Y}(Y \vertspace D)\lvert_{F}^{\simple}]^{\vir} \times \xi^![\Mbar_{\Gamma_Z}(Z \vertspace D_0+D_\infty)\lvert_{F}^{\rubber}]^{\vir}
	\end{align*}
	by \cite[Theorem 6.2 (c)]{Fulton}. Comparing the usual perfect obstruction theory of $\Mbar_{\Gamma_Z}(\bP^1 \vertspace 0+\infty)^{\rubber}$ with the one pulled back from $\Mbar_{\Gamma_Z}(Z \vertspace D_0+D_\infty)\lvert_{F}^{\rubber}$ we find
	\begin{equation*}
		\xi^![\Mbar_{\Gamma_Z}(Z \vertspace D_0+D_\infty)\lvert_{F}^{\rubber}]^{\vir} = \uplambda_{\mr{top}}^2 \cap [\Mbar_{\Gamma_Z}(\bP^1 \vertspace 0+\infty)^{\rubber}]^{\vir}
	\end{equation*}
	where $\uplambda_{\mr{top}}$ is the top Chern class of the Hodge bundle. The statement of the \namecref{prop: higher genus vanishing} then follows from the fact that $\uplambda_{\mr{top}}^2 = 0$ if any of the vertices in $\Gamma_Z$ carries a non-zero genus label and $\uplambda_{\mr{top}}^2 = 1$ otherwise.
\end{proof}

\subsubsection{Final vanishing}
\label{sec: composite 3rd vanishing}
Having proven \Crefrange{prop: composite unique Dinfty leg}{prop: higher genus vanishing} we are finally able to deduce \Cref{prop: composite contrib vanishes}.

\begin{proof}[Proof of \Cref{prop: composite contrib vanishes}]
	According to \Cref{prop: composite unique Dinfty leg} and \labelcref{prop: composite unique D0 leg} the only splitting types $\Gamma$ whose associated summand on the left-hand side of \eqref{eq: comp locus contrib rewritten} may be non-zero consist of $Z$-vertices with exactly one adjacent edge and leg. Moreover, by \Cref{prop: higher genus vanishing} every such vertex has to carry a vanishing genus label. Therefore, the virtual dimension of $\Mbar_{\Gamma_Z}(\bP^1 \vertspace 0+\infty)^{\rubber}$ is
	\begin{equation*}
		\sum_{v\in V(\Gamma_Z)} (2g_v-2) + |E(\Gamma)| + |L(\Gamma)| - 1 = -1
	\end{equation*}
	and so the cycle $[\Mbar_{\Gamma}]^{\vir}$ vanishes by \Cref{prop: higher genus vanishing}.
\end{proof}

\section{Step II: reducing interior to contact insertions}
\label{sec: reduction}
\subsection{The comparison statement} Throughout this section we will assume that $(S \vertspace D_1+D_2)$ is a logarithmic Calabi--Yau surface with boundary the union of two transversally intersecting smooth irreducible curves $D_1,D_2$. Moreover, we fix an effective curve class $\upbeta$ in $S$ and assume
\begin{itemize}
\item $D_i \cdot \upbeta >0$, $i\in\{1,2\}$,
\item $D_2 \cdot D_2\geq 0$.
\end{itemize}
Especially, in this section we do not necessarily assume $(S\vertspace D_1)$ to deform to a toric pair. We set $Y\coloneqq \Tot \cO_S (-D_2)$ and write $D$ for the preimage of $D_1$ under the projection morphism.

As before we fix a contact datum $\mathbf{c}=(c_1,\ldots,c_n)$ with $D$ where $c_i>0$ for all $i\in\{1,\ldots,n\}$. We set
\begin{equation}
	\label{eq: defn log GW Y D}
	\LGW{g,\mathbf{c},\upbeta}{Y \vertspace D}{\prod_{i=1}^n \mr{ev}_i^{\cstar} \big(\uppi^{\cstar} \mr{pt}\big)} = \int_{[\Mbar_{g,\mathbf{c},\upbeta}(Y \vertspace D)]^{\vir}} \prod_{j=1}^{n} \mr{ev}_j^{\cstar}\big(\uppi^{\cstar}\big[\mr{pt}_{D_1}\big]\big)\,.
\end{equation}
Now denote by
\begin{equation*}
	\mathbf{\tilde{c}}=(0,\ldots,0,c_1,\ldots,c_n)
\end{equation*}
the result of adding $m\leq n$ interior markings. Similarly to \eqref{eq: defn log GW Y D} we define $\LGW{g,\mathbf{\tilde{c}},\upbeta}{Y \vertspace D}{\prod_{i=1}^n \mr{ev}_i^{\cstar} \big(\uppi^{\cstar} \mr{pt}\big)}\in \bQ$ by imposing point conditions at the first $n$ markings as in \eqref{eq: def log GW Y D}. Note that this especially means that there is a point condition at all $m$ interior markings. We have the following comparison statement which reduces computations to the case without any interior markings.

\begin{prop}
	\label{thm: comparison point conditions}
	We have
	\begin{equation}
		\label{eq: comparison point conditions}
		\LGW{g,\mathbf{\tilde{c}},\upbeta}{Y \vertspace D}{\Pi_{i=1}^{n}\mr{ev}_i^{\cstar}(\uppi^{\cstar}\mr{pt})} = m! \left( \textstyle{\prod_{i=0}^{m-1}}c_{n-i}\right) \cdot \LGW{g,\mathbf{c},\upbeta}{Y \vertspace D}{\Pi_{i=1}^n \mr{ev}_i^{\cstar}(\uppi^{\cstar}\mr{pt})}\,.
	\end{equation}
\end{prop}

Together with the special case $m=0$ of \Cref{prop: log threefold to descedant}, which we have already proven in the last section, the above reduction result implies the general case $m\geq 0$ of \Cref{prop: log threefold to descedant} as well. The rest of this section is hence dedicated to the proof of \Cref{thm: comparison point conditions}.

\subsection{Degeneration}
To prove the main result of this section we will use a degeneration argument similar to \cite{vGGR19}. Since by now this approach has become a well-established technique \cite{TY20:HigherGenusRelOrb,BNTY23:LocOrbSNC,GNS23:bicyclic,CMS24:QuasiMapDegen,BFGW21:HAE} we economise on details.

Recall the definition of the line bundle $Z\rightarrow P$ from equation \eqref{eq: proj compl Z} and \eqref{eq: proj compl P} (where we identify $S=\defS$). We defined $Z$ to be the projective completion of $N_D Y$ whose zero and infinity section we denote by $D_0$ and $D_{\infty}$ respectively. Let us consider the degeneration to the normal cone of $D$ in $Y$:
\begin{equation*}
	\cY \longrightarrow \Aaff{1}\,.
\end{equation*}
$\cY$ is a line bundle over the degeneration to the normal cone of $D_1$ in $S$. While the general fibre of $\cY$ is $Y$, the special fibre $\cY_0 = Y \sqcup_D Z$ is the result of gluing $Y$ along $D$ to $Z$ along the zero section $D_0$. Now let us write $\cD$ for the closure of $D\times (\Aaff{1} - \{0\})$ in $\cY$ and equip $\cY$ with the divisorial logarithmic structure with respect to $Y+Z+\cD$. The resulting logarithmic scheme has general fibre $(Y \vertspace D)$ and its special fibre is obtained by gluing
\begin{equation*}
	(Y \vertspace D) \qquad \text{and} \qquad (Z \vertspace D_0+D_\infty)\,.
\end{equation*}
Finally, for $F$ a fibre of $D\rightarrow D_1$ we denote by $\cF$ the closure of $F \times (\Aaff{1}-\{0\})$ in $\cY$. The intersection $\cF_0$ of $\cF$ with the special fibre $\cY_0$ is a fibre of $D_\infty\rightarrow D_1$.

\subsection{Decomposition}
Composition with the blowup morphism $\cY \rightarrow Y\times \Aaff{1}$ produces a pushforward
\begin{equation*}
	\uprho : \Mbar_{g,\mathbf{\tilde{c}},\upbeta}(\cY_0) \longrightarrow \Mbar_{g,\mathbf{\tilde{c}},\upbeta}(Y \vertspace D)
\end{equation*}
and by the conservation of number principle we have
\begin{equation}
	\label{eq: conservation of number principle}
	[\Mbar_{g,\mathbf{\tilde{c}},\upbeta}(Y \vertspace D)\lvert_{F}]^{\vir} = (\uprho_{\mc{F}_0})_\cstar [\Mbar_{g,\mathbf{\tilde{c}},\upbeta}(\cY_0)\lvert_{\mc{F}_0}]^{\vir}
\end{equation}
where we introduced the fibre products
\begin{equation*}
	\begin{tikzcd}
		\Mbar_{g,\mathbf{\tilde{c}},\upbeta}(Y \vertspace D)\lvert_{F} \arrow[r, hook] \arrow[d] \ar[rd,phantom,"\square",start anchor=center,end anchor=center] & \Mbar_{g,\mathbf{\tilde{c}},\upbeta}(Y \vertspace D) \arrow[d, "\Pi_{i=1}^n\mr{ev}_i"] & & \Mbar_{g,\mathbf{\tilde{c}},\upbeta}(\cY_0)\lvert_{\cF_0} \arrow[r, hook] \arrow[d] \ar[rd,phantom,"\square",start anchor=center,end anchor=center] & \Mbar_{g,\mathbf{\tilde{c}},\upbeta}(\cY_0) \arrow[d, "\Pi_{i=1}^n\mr{ev}_i"]\\
		F^n \arrow[r,hook] & Y^n & & \cF_0^n \arrow[r,hook] & \cY_0^n\,.
	\end{tikzcd}
\end{equation*}
and pulled back the respective virtual fundamental classes along horizontal arrows. The pushforward of the left-hand side of \eqref{eq: conservation of number principle} to $\Mbar_{g,\mathbf{\tilde{c}},\upbeta}(Y \vertspace D)$ gives
\begin{equation*}
	\Pi_{i=1}^{n} \mr{ev}_i^{\cstar}\big(\uppi^{\cstar}\big[\mr{pt}_{S}\big]\big) \cdot \Pi_{j=m+1}^{n} \mr{ev}_j^{\cstar}\big(\uppi^{\cstar}\big[\mr{pt}_{D_1}\big]\big) \cap [\Mbar_{g,\mathbf{\tilde{c}},\upbeta}(Y \vertspace D)]^{\vir}
\end{equation*}
and so pushing forward to a point returns the Gromov--Witten invariant we are after:
\begin{equation*}
	\LGW{g,\mathbf{\tilde{c}},\upbeta}{Y \vertspace D} {\Pi_{i=1}^n\mr{ev}_i^{\cstar} (\pi^{\cstar}\mr{pt})}\,.
\end{equation*}

We combine this with the decomposition formula \cite[Theorem 5.4]{ACGS20:Decomposition}
\begin{equation}
	\label{eq: decomposition}
	[\Mbar_{g,\mathbf{\tilde{c}},\upbeta}(\cY_0)\lvert_{\mc{F}_0}]^{\vir} = \sum_{\uptau} \frac{m_{\uptau}}{ | \Aut \uptau | } \upiota_\cstar [\Mbar_{\uptau}(\cY_0)\lvert_{\mc{F}_0}]^{\vir}
\end{equation}
where $\upiota$ denotes the inclusion
\begin{equation*}
	\Mbar_{\uptau}(\cY_0)\lvert_{\mc{F}_0} \hookrightarrow \Mbar_{g,\mathbf{\tilde{c}},\upbeta}(\cY_0)\lvert_{\mc{F}_0}
\end{equation*}
and the sum runs over all rigid tropical types $\uptau$ of tropical maps to the polyhedral complex
\begin{equation*}
	\begin{tikzpicture}[baseline=(current  bounding  box.center)]
		\draw (-3,0) node{$\Sigma \, :$};

		\draw (-2,0) -- (0,0);

		\draw[->] (0,0) -- (3,0);

		\draw[fill=black] (-2,0) circle[radius=2pt];
		\draw (-2,0) node[below]{$\upsigma_Y$};

		\draw[fill=black] (0,0) circle[radius=2pt];
		\draw (0,0) node[below]{$\upsigma_Z$};
	\end{tikzpicture}
\end{equation*}
which is the fibre over $1\in \mbb{R}_{\geq 0}$ of the tropicalisation $\Sigma(\cY) \rightarrow \Sigma (\Aaff{1}) =  \mbb{R}_{\geq 0}$.

By the discussion in \cite[Section 5.1]{ACGS20:Decomposition} $\uptau$ can only be rigid if all vertices of the underlying graph map to either one of the vertices $\upsigma_X$ or $\upsigma_Z$. Hence, $\uptau$ is uniquely reconstructed from the data of its underlying decorated bipartite graph which is exactly the data of a splitting type as introduced in \Cref{defn: splitting type}. Hence, to be closer to the notation used in \Cref{sec: composite locus} let us write $\Gamma$ for the splitting type underlying $\uptau$ and $\Gamma_Y$, $\Gamma_Z$ for the type of stable logarithmic maps to $(Y \vertspace D)$, respectively $(Z \vertspace D_0+D_\infty)$, obtained by cutting $\Gamma$ along edges.

We remark that only such rigid tropical types can contribute non-trivially to the sum \eqref{eq: decomposition} for which all marking legs are attached to a $Z$-vertex. Indeed, contact markings necessarily have to be adjacent to a $Z$-vertex and if an interior marking is carried by a $Y$-vertex then $\Mbar_{\uptau}(\cY_0)\lvert_{\mc{F}_0}$ is empty (where we crucially use our assumption that $m\leq n$). Moreover, the factor $m_{\uptau}$ in the decomposition formula \eqref{eq: decomposition} equates to $\mr{lcm}(d_e)$ where $d_e$ denotes contact order of an edge $e$.

Since we are in the special setting where $\Sigma$ is of dimension one the discussion in \cite[Section 7]{Gro23:RmksGluing} applies. There is a morphism $\upnu$ with target the fibre product
\begin{equation}
	\label{eq: defn Mbar Gamma}
	\begin{tikzcd}
		\Mbar_{\uptau}(\cY_0)\lvert_{\mc{F}_0} \arrow[r, "\upnu"] & \Mbar_\Gamma \arrow[r] \arrow[d] \ar[rd,phantom,"\square"] & \Mbar_{\Gamma_Y}(Y \vertspace D)\times \Mbar_{\Gamma_Z}(Z \vertspace D_0 + D_\infty)\lvert_{\mc{F}_0} \arrow[d] \\
		& D_0^{E(\Gamma)} \arrow[r, "\Delta"] & D_0^{E(\Gamma)} \times D_0^{E(\Gamma)}
	\end{tikzcd}
\end{equation}
which is finite and satisfies
\begin{equation}
	\label{eq: decomposition 2}
	\upnu_\cstar [\Mbar_{\uptau}(\cY_0)\lvert_{\mc{F}_0}]^{\vir} = \frac{\Pi_e \, d_e}{\mr{lcm}(d_e)} \, \Delta^! \big(
	 [\Mbar_{\Gamma_Y}(Y \vertspace D)]^{\vir} \times [\Mbar_{\Gamma_Z}(Z \vertspace D_0 + D_\infty)\lvert_{\mc{F}_0}]^{\vir} \big)\,.
\end{equation}
There is a gluing morphism $\uptheta: \Mbar_\Gamma \rightarrow \Mbar_{g,m+n,\upbeta}(Y)$ such that
\begin{equation}
	\label{eq: gluing morph commutes}
	\begin{tikzcd}
		\Mbar_{\uptau}(\cY_0)\lvert_{\mc{F}_0} \arrow[r, "\upnu"] \arrow[d, "\upiota"] & \Mbar_{\Gamma} \arrow[d,"\uptheta"] \\
		\Mbar_{g,\mathbf{\tilde{c}},\upbeta}(\cY_0)\lvert_{\mc{F}_0} \arrow[r,"\upzeta\circ \uprho_{\mc{F}_0}"] & \Mbar_{g,m+n,\upbeta}(Y)
	\end{tikzcd}
\end{equation}
commutes where $\upzeta$ denotes the forgetful morphism
\begin{equation*}
	\upzeta : \Mbar_{g,\mathbf{\tilde{c}},\upbeta}(Y \vertspace D)\lvert_{F} \longrightarrow \Mbar_{g,m+n,\upbeta}(Y)\,.
\end{equation*}
To summarize the discussion so far we combine equation \eqref{eq: decomposition} and \eqref{eq: decomposition 2} and use the fact that diagram \eqref{eq: gluing morph commutes} commutes to deduce
\begin{equation}
	\label{eq: degeneration result}
	\upzeta_\cstar [\Mbar_{g, \mathbf{\tilde{c}} ,\upbeta}(Y \vertspace D)\lvert_{F}]^{\vir} = \sum_{\Gamma} \frac{\Pi_e \, d_e}{|\Aut \Gamma|} \uptheta_\cstar [\Mbar_\Gamma]^{\vir}
\end{equation}
where we introduced the notation
\begin{equation*}
	[\Mbar_\Gamma]^{\vir} \coloneqq \Delta^! \big(
	[\Mbar_{\Gamma_Y}(Y \vertspace D)]^{\vir} \times [\Mbar_{\Gamma_Z}(Z \vertspace D_0 + D_\infty)\lvert_{\mc{F}_0}]^{\vir} \big)\,.
\end{equation*}

\subsection{Vanishing}
In this section we will argue that for most splitting types $\Gamma$ the cycle $[\Mbar_\Gamma]^{\vir}$ vanishes. We note that our setup is almost the same as in \Cref{sec: composite locus}. The only notable differences are as follows.
\begin{itemize}
	\item $\Mbar_{\Gamma_Y}(Y \vertspace D)$ parametrise logarithmic stable maps to $(Y \vertspace D)$. In \Cref{sec: composite locus} we considered the simple fixed locus of this moduli stack under some torus action.

	\item $\Mbar_{\Gamma_Z}(Z \vertspace D_0 + D_\infty)$ parametrises logarithmic stable maps to $(Z \vertspace D_0 + D_\infty)$ while in \Cref{sec: composite locus} the target was non-rigid $(Z \vertspace D_0 + D_\infty)$.

	\item $\Gamma_Z$ contains interior markings.
\end{itemize}
We note that all arguments in \Cref{sec: composite locus} are essentially insensitive to first two differences which allows us to prove the vanishing of $[\Mbar_\Gamma]^{\vir}$ in similar steps. Only the third point requires some attention.

\subsubsection{First vanishing: fibre classes}
We start with the analogue of \Cref{lem: GNS23 lemma}.

\begin{lem}
	\label{lem: Zvertices fibre class}
	We have $[\Mbar_{\Gamma_Z}(Z \vertspace D_0 + D_\infty)\lvert_{\mc{F}_0}]^{\vir}=0$ unless $\Gamma_Z$ only consists of vertices decorated with a curve class which is a multiple of a fibre class.
\end{lem}
\begin{proof}
	Let us write $\upbeta_Z$ for the overall curve class carried by $\Gamma_Z$ and suppose its projection $p_\cstar \upbeta_Z$ along the morphism $p:P\rightarrow D_1$ is non-zero. In other words we assume $p_\cstar \upbeta_Z = d[D_1]$ for some $d>0$. We claim that in this case
	\begin{equation*}
		[\Mbar_{\Gamma_Z}(Z \vertspace D_0 + D_\infty)]^{\vir}=0\,.
	\end{equation*}
	To prove this we use an argument from \cite[Section 4]{vGGR19}. Set $k=m+n+|E(\Gamma)|$ and consider the morphism
	\begin{equation*}
		\Mbar_{\Gamma_Z}(Z \vertspace D_0 + D_\infty) \longrightarrow \Mbar_{g_Z,k, d[D_1]} (D_0)
	\end{equation*}
	obtained by forgetting the logarithmic structure and composing with the projection $Z \rightarrow D_0$. Write $\hhh_2(P)^+$ and $\hhh_2(D_1)^+$ for the semigroups of effective curve classes in $P$, respectively $D_1$ and denote by $\mf{M}_{g,n,\hhh}^{(\mr{log})}$ the Artin stack of prestable (logarithmic) curves with irreducible components weighted by an element in a semigroup $\hhh$. We factor the projection morphism through
	\begin{equation*}
		\begin{tikzcd}
			\Mbar_{\Gamma_Z}(Z \vertspace D_0 + D_\infty) \arrow[r, "u"] \arrow[dr] & \widetilde{\modulifont{M}}_{\Gamma_Z} \arrow[r] \arrow[d] \ar[rd,phantom,"\square",start anchor=center,end anchor=center] & \Mbar_{g_Z,k, d[D_1]} (D_0) \arrow[d] \\
			& \mf{M}_{g_Z,k,\hhh_2(P)^+}^{\mr{log}} \arrow[r, "v"] & \mf{M}_{g_Z,k,\hhh_2(D_1)^+}
		\end{tikzcd}
	\end{equation*}
	Then by \cite[Theorem 4.1]{vGGR19} we have\footnote{In the statement of \cite[Theorem 4.1]{vGGR19} van Garrel, Graber and Ruddat assume a positivity property for the logarithmic tangent sheaf of the target which guarantees that $u$ is smooth. In our case at hand this positivity condition is in general not satisfied and $u$ is only virtually smooth in the sense of \cite[Definition~3.4]{Man12:VirtPush}. The latter is however sufficient for our purposes thanks to \cite[Corollary~4.9]{Ma11:VirtPull}.}
	\begin{equation*}
		[\Mbar_{\Gamma_Z}(Z \vertspace D_0 + D_\infty)]^{\vir} = u^!_{\mf{E}_{\Mbar_{\Gamma_Z}(Z \vertspace D_0 + D_\infty)/\widetilde{\modulifont{M}}_{\Gamma_Z}}} v^! [\Mbar_{g_Z,k, d[D_1]} (D_0)]^{\vir}\,.
	\end{equation*}
	Our claim that $[\Mbar_{\Gamma_Z}(Z \vertspace D_0 + D_\infty)]^{\vir}=0$ directly follows from the fact that
	\begin{equation*}
		[\Mbar_{g_Z,k, d[D_1]} (D_0)]^{\vir} = 0
	\end{equation*}
	since $D_0 \cong \Tot \cO_{\bP^1}(-2)$ carries a nowhere vanishing holomorphic two-form and $d>0$.
\end{proof}

\subsubsection{Second vanishing: star shaped graphs} In this section we apply the same arguments we used in \Cref{sec: composite 2nd vanishing} in order to constrain the shape of $\Gamma$ to star shaped graphs.

\begin{notn}
	We introduce the following notation for labelling sets:
	\begin{itemize}
		\item Interior markings with a point condition $I_{\mr{pt}} \coloneqq \{1,\ldots,m\}$;
		\item Contact markings with a point condition $J_{\mr{pt}} \coloneqq \{m+1,\ldots,n\}$;
		\item Contact markings without an insertion $J_{\mathbbm{1}} \coloneqq \{n+1,\ldots,m+n\}$.
	\end{itemize}
\end{notn}

We may assume that all $Z$-vertices of $\Gamma$ are decorated with a curve class which is a multiple of a fibre class. This has the following consequence.
\begin{lem}
	\label{prop: unique Dinfinity leg}
	We have $[\Mbar_{\Gamma_Z}(Z \vertspace D_0 + D_\infty)\lvert_{\mc{F}_0}]^{\vir} = 0$ unless every $Z$-vertex
	\begin{itemize}
		\item either carries two legs, one labelled in $I_{\mr{pt}}$ and one in $J_{\mathbbm{1}}$;

		\item or carries a single leg with label in $J_{\mr{pt}}$.
	\end{itemize}
\end{lem}

\begin{proof}
	It suffices to show that every $Z$-vertex is carrying at most one leg with label in $I_{\mr{pt}} \cup J_{\mr{pt}}$. Indeed, in this case the number of vertices in $\Gamma_Z$ is bounded from below by $|I_{\mr{pt}}| + |J_{\mr{pt}}| = n$. On the other hand it is also bounded from above by the number of contact markings $n$. We deduce that there have to be exactly $n$ vertices each carrying one of the legs labelled by $J_{\mathbbm{1}}\cup J_{\mr{pt}}$. Since we assumed that every vertex carries at most one leg with label in $I_{\mr{pt}} \cup J_{\mr{pt}}$, the $m$ interior markings are necessarily distributed among the vertices carrying a leg labelled in $J_{\mathbbm{1}}$.

	So suppose there is a $Z$-vertex with more than one leg with label in $I_{\mr{pt}} \cup J_{\mr{pt}}$. Since the evaluation morphism of an interior marking has target $Z$ let us first define a substack with confined evaluations:
	\begin{equation*}
		\Nbar_{\Gamma_Z} \coloneqq \Mbar_{\Gamma_Z}(Z \vertspace D_0 + D_\infty) \times_{Z^n} D_\infty^n\,.
	\end{equation*}
	As every $Z$-vertex is decorated with a curve class which is a multiple of a fibre class and we assume there is a vertex with two legs labelled in $I_{\mr{pt}} \cup J_{\mr{pt}}$ the resulting evaluation morphism
	\begin{equation*}
		\Nbar_{\Gamma_Z} \longrightarrow D_\infty^n
	\end{equation*}
	factors through $D_\infty^{n-1} \rightarrow D_\infty^n$. We obtain the following commuting diagram with Cartesian squares:
	\begin{equation*}
		\begin{tikzcd}
			\Mbar_{\Gamma_Z}(Z \vertspace D_0 + D_\infty)\lvert_{\mc{F}_0} \arrow[r] \arrow[d] \ar[rd,phantom,"\square",start anchor=center,end anchor=center] & \Nbar_{\Gamma_Z} \arrow[d] \arrow[r] \ar[rdd,phantom,"\square",start anchor=center,end anchor=center] & \Mbar_{\Gamma_Z}(Z \vertspace D_0 + D_\infty) \arrow[dd,"\Pi_{i=1}^n \mr{ev}_i"]\\
			\mc{F}_0^{n-1} \arrow[d] \arrow[r] \ar[rd,phantom,"\square",start anchor=center,end anchor=center] & D_\infty^{n-1} \arrow[d] & \\
			\mc{F}_0^{n} \arrow[r] & D_\infty^n \arrow[r] & Z^{n}\,.
		\end{tikzcd}
	\end{equation*}
	The vanishing of $[\Mbar_{\Gamma_Z}(Z \vertspace D_0 + D_\infty)\lvert_{\mc{F}_0}]^{\vir}$ then immediately follows from an excess bundle calculation as the one we did in \Cref{prop: composite unique Dinfty leg}.
\end{proof}

Moreover, by our assumption that all $Z$-vertices carry curve classes which are multiples of a fibre class the evaluation morphism associated to edges
\begin{equation*}
	\Mbar_{\Gamma_Z}(Z \vertspace D_0 + D_\infty)\lvert_{\mc{F}_0} \longrightarrow D_0^{E(\Gamma)}
\end{equation*}
factors through $F^{E(\Gamma)} \hookrightarrow D_0^{E(\Gamma)}$ where we write $F$ for the image of $\cF_0$ under $Z\rightarrow D_0$. We combine this observation with the fact that stable maps to $Y$ factor through the zero section $S\hookrightarrow Y$ in order to constrain the image of $Z$-vertices to the fibre of $Z\rightarrow D_0$ over the point $p = F \cap D_1$. Indeed, as we just argued the evaluation morphism
\begin{equation*}
	\Mbar_{\Gamma_Y}(Y \vertspace D) \longrightarrow D^{E(\Gamma)}
\end{equation*}
factors through $D_1^{E(\Gamma)} \hookrightarrow D^{E(\Gamma)}$. Writing $\Mbar_{\Gamma_Y}(Y \vertspace D)\lvert_{F}$ for the fibre product $\Mbar_{\Gamma_Y}(Y \vertspace D)\times_{D^{E(\Gamma)}} F^{E(\Gamma)}$ we find that
\begin{equation}
	\label{eq: diagram restrict image to fibre}
	\begin{tikzcd}
		\Mbar_{\Gamma_Y}(Y \vertspace D)\lvert_{F}\,\times\, \Mbar_{\Gamma_Z}(\bP^1 \vertspace 0+\infty)\lvert_{\infty} \arrow[d] \arrow[r] \ar[rd,phantom,"\square",start anchor=center,end anchor=center] & \Mbar_{\Gamma_Y}(Y \vertspace D) \times \Mbar_{\Gamma_Z}(Z \vertspace D_0+D_\infty)\lvert_{\cF_0} \arrow[d] \\
		p \arrow[d] \arrow[r] \ar[rd,phantom,"\square",start anchor=center,end anchor=center] & D_1^{E(\Gamma)} \times F^{E(\Gamma)} \arrow[d] \\
		D_0^{E(\Gamma)} \arrow[r, "\quad~~\Delta"] & D_0^{E(\Gamma)} \times D_0^{E(\Gamma)} \,.
	\end{tikzcd}
\end{equation}
commutes and both squares are Cartesian. An excess intersection argument as in \Cref{prop: composite unique D0 leg} gives

\begin{lem}
	\label{prop: unique D0 legs}
	We have $[\Mbar_\Gamma]^{\vir}=0$ if there is a $Z$-vertex with more than one edge.\qed
\end{lem}

Since all squares in \eqref{eq: diagram restrict image to fibre} are Cartesian, a comparison with \eqref{eq: defn Mbar Gamma} gives
\begin{equation*}
	\label{eq: isom fibre squares}
	\Mbar_{\Gamma} \cong \Mbar_{\Gamma_Y}(Y \vertspace D)\lvert_{F}\,\times\, \Mbar_{\Gamma_Z}(\bP^1 \vertspace 0+\infty)\lvert_{\infty}\,.
\end{equation*}
The following \namecref{prop: higher genus vanishing 2} comparing the virtual fundamental classes of the above moduli stacks can be proven with the same arguments we used to show \Cref{prop: higher genus vanishing}.

\begin{lem}
	\label{prop: higher genus vanishing 2}
	If $\Gamma$ contains a $Z$-vertex $v$ with $g_v>0$ the cycle $[\Mbar_\Gamma]^{\vir}$ vanishes. Otherwise, it equates to
	\begin{flalign*}
		&&[\Mbar_{\Gamma_Y}(Y \vertspace D)\lvert_{F}]^{\vir} \times [\Mbar_{\Gamma_Z}(\bP^1 \vertspace 0+\infty)\lvert_{\infty}]^{\vir}\,.&&\qed
	\end{flalign*}
\end{lem}

Let us summarise the results of this section.

\begin{prop}
	\label{prop: summary vanishing}
	The cycle $[\Mbar_\Gamma]^{\vir}$ vanishes unless the rigid tropical type $\uptau$ associated to $\Gamma$ is of the form
	
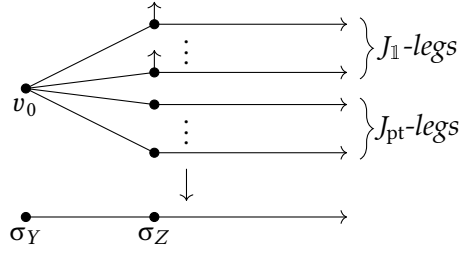
\begin{equation*}
		\begin{tikzpicture}[baseline=(current  bounding  box.center),scale=0.85]
			\draw[fill=black] (-2,2) circle[radius=2pt];
			\draw (-2,2) node[below]{$v_0$};

			\draw (-2,2) -- (0,1);
			\draw[fill=black] (0,1) circle[radius=2pt];
			\draw[->] (0,1) -- (3,1);

			\draw (0.5,1.44) node{$\vdots$};

			\draw (-2,2) -- (0,1.75);
			\draw[fill=black] (0,1.75) circle[radius=2pt];
			\draw[->] (0,1.75) -- (3,1.75);

			\draw (-2,2) -- (0,2.25);
			\draw[fill=black] (0,2.25) circle[radius=2pt];
			\draw[->] (0,2.25) -- (3,2.25);
			\draw[->] (0,2.25) -- (0,2.6);

			\draw (0.5,2.69) node{$\vdots$};

			\draw (-2,2) -- (0,3);
			\draw[fill=black] (0,3) circle[radius=2pt];
			\draw[->] (0,3) -- (3,3);
			\draw[->] (0,3) -- (0,3.35);


			\draw[->] (0.5,0.75) -- (0.5,0.25);

			\draw (-2,0) -- (0,0);

			\draw[->] (0,0) -- (3,0);

			\draw[fill=black] (-2,0) circle[radius=2pt];
			\draw (-2,0) node[below]{$\upsigma_Y$};

			\draw[fill=black] (0,0) circle[radius=2pt];
			\draw (0,0) node[below]{$\upsigma_Z$};

			\draw [decorate,decoration={brace,amplitude=5pt,mirror}]
			(3.2,0.9) -- (3.2,1.85) node[midway,xshift=2em]{$J_{\mr{pt}}$-legs};
			\draw [decorate,decoration={brace,amplitude=5pt,mirror}]
			(3.2,2.15) -- (3.2,3.1) node[midway,xshift=2em]{$J_{\mr{\mathbbm{1}}}$-legs};
		\end{tikzpicture}
	\end{equation*}
	where $v_0$ is decorated with $g_{v_0}=g$, $\upbeta_{v_0}=\upbeta$ and all other vertices are decorated with multiples of fibre classes. Moreover, in this case
	\begin{flalign*}
		&&[\Mbar_{\Gamma}]^{\vir} = [\Mbar_{\Gamma_Y}(Y \vertspace D)\lvert_{F}]^{\vir} \times [\Mbar_{\Gamma_Z}(\bP^1 \vertspace 0+\infty)\lvert_{\infty}]^{\vir}\,. && \qed
	\end{flalign*}
\end{prop}

\subsection{Analysis of remaining terms}
Now our discussion will deviate from the analysis in \Cref{sec: composite locus} because as opposed to \Cref{sec: composite 3rd vanishing} we will show that all remaining contributions characterised in \Cref{prop: summary vanishing} do contribute to the Gromov--Witten invariant and will exactly yield the right-hand side of equation \eqref{eq: comparison point conditions}.

Let $\Gamma$ be as in \Cref{prop: summary vanishing}. We fix a labelling of the markings of $\Mbar_{\Gamma_Y}(Y \vertspace D)\lvert_{F}$ by fixing a bijection $\{1,\ldots,n\} \cong E(\Gamma)$ where we impose that the $i$th edge is adjacent to the $Z$-vertex carrying the $(m+i)$th leg. With this assignment we can naturally identify $\Gamma_Y$ with the data $(g,\mathbf{c},\upbeta)$.

Our first goal is to simplify
\begin{equation*}
	\uptheta_\cstar [\Mbar_\Gamma]^{\vir} = \uptheta_\cstar\big( [\Mbar_{g,\mathbf{c},\upbeta}(Y \vertspace D)\lvert_{F}]^{\vir} \times [\Mbar_{\Gamma_Z}(\bP^1 \vertspace 0+\infty)\lvert_{\infty}]^{\vir}\big)\,.
\end{equation*}
Under $\uptheta$ all images of stable maps which lie in a fibre of $P\rightarrow D_1$ get contracted to a point. Hence, the morphism factors into
\begin{equation*}
	\begin{tikzcd}
		\Mbar_{g,\mathbf{c},\upbeta}(Y \vertspace D)\lvert_{F}\, \times\,\Mbar_{\Gamma_Z}(\bP^1 \vertspace 0+\infty)\lvert_{\infty} \arrow[r, "\mr{Id}\times \uppi"] \arrow[rr, start anchor={[yshift=0.4em]south east},  end anchor={[yshift=0.4em]south west}, swap, bend right = 10, "\uptheta"] &
		\Mbar_{g,\mathbf{c},\upbeta}(Y \vertspace D)\lvert_{F} \,\times\, \big(\Mbar_{0,3}\big)^{m} \arrow[r, "\upphi"] &
		\Mbar_{g,m+n,\upbeta}(Y)
	\end{tikzcd}
\end{equation*}
where $\uppi$ is the forgetful morphism which only remembers the stabilised domain curve and $\upphi$ is a morphism which forgets the logarithmic structure and attaches a contracted rational component at the last $m$ markings.

Consequently, a $Z$-vertex carrying a leg labelled with $i+m \in J_{\mathbbm{1}}$ contributes with a factor
\begin{equation*}
	\deg [\Mbar_{0,((c_i,0),(0,c_i)),c_i}(\bP^1 \vertspace 0+\infty)\lvert_{\infty}]^{\vir} = \deg [\Mbar_{0,((c_i,0),(0,c_i)),c_i}(\bP^1 \vertspace 0+\infty)]^{\vir} = \frac{1}{c_i}
\end{equation*}
where the denominator is due to the degree $c_i$ automorphism group acting on the unique degree $c_i$ cover $\bP^1 \rightarrow \bP^1$ fully ramified over $0$ and infinity. Similarly, a $Z$-vertex with leg labelled by $i+m \in  J_{\mathrm{pt}}$ contributes
\begin{equation*}
	\deg [\Mbar_{0,(0,(c_i,0),(0,c_i)),c_i}(\bP^1 \vertspace 0+\infty)\lvert_{\infty}]^{\vir} = \deg \big( \mr{ev}_1^{\cstar}(\mr{pt})\cap [\Mbar_{0,(0,(c_i,0),(0,c_i)),c_i}(\bP^1 \vertspace 0+\infty)]^{\vir} \big) = 1
\end{equation*}
because the additional marking kills the automorphism group acting on the domain. We obtain
\begin{equation}
	\label{eq: contribution splitting type final}
	\uptheta_\cstar [\Mbar_\Gamma]^{\vir} = \left(\textstyle{\prod_{i=1}^{n-m}} \tfrac{1}{c_i}\right) \cdot \upphi_\cstar \!\left([\Mbar_{g, \mathbf{c} ,\upbeta}(Y \vertspace D)\lvert_{F}]^{\vir} \times [\Mbar_{0,3}]^{\times m}\right)\,.
\end{equation}

Note that a splitting type $\Gamma$ as in \Cref{prop: summary vanishing} is uniquely determined by the bijection $I_{\mr{pt}} \cong J_{\mathbbm{1}}$ it induces. This means there are exactly $m!$ such splitting types each one contributing with \eqref{eq: contribution splitting type final} to the overall Gromov--Witten invariant. Thus, plugging the last equation into \eqref{eq: degeneration result} we find
\begin{equation*}
	\upzeta_\cstar [\Mbar_{g, \mathbf{\tilde{c}} ,\upbeta}(Y \vertspace D)\lvert_{F}]^{\vir} = m! \left(\textstyle{\prod_{i=0}^{m-1}}c_{n-i}\right) \cdot \upphi_\cstar \!\left([\Mbar_{g, \mathbf{c} ,\upbeta}(Y \vertspace D)\lvert_{F}]^{\vir} \times [\Mbar_{0,3}]^{\times m}\right)\,.
\end{equation*}
Pushing the identity forward to a point yields formula \eqref{eq: comparison point conditions} which closes the proof of \Cref{thm: comparison point conditions}.

\section{Applications}
\label{sec: application}

\subsection{Topological Vertex}
\label{sec: Top Vert}
Open Gromov--Witten invariants can be computed efficiently using the topological vertex method \cite{AKMV05:TopVert,LLLZ09:MathTopVert}. Hence, formula \eqref{eq:log-open} equating logarithmic Gromov--Witten invariants --- which are usually rather cumbersome to compute --- against open ones presents a new venue for calculating the former.

\subsubsection{The general recipe}
We refer to \cite[Section~2]{Yu23:BPS} or \cite[Section~6]{BBvG2} for a general summary of the topological vertex method which allows to determine the open Gromov--Witten invariants of a toric triple just from the information of its toric diagram. Let $(X,L,\msf{f})$ be the outcome of \Cref{constr: toric triple} for some two-component Looijenga pair $(Y\vertspace D_1 + D_2)$. We will write
\begin{equation}
	\label{eq: disconnected open invariant}
	\cW_{\upmu}(X,L,\msf{f})(q,Q)
\end{equation}
for the generating series of disconnected open Gromov--Witten invariants of $(X,L,\msf{f})$ in representation basis where the winding around $L$ is encoded in a partition $\upmu$. In the special case of a single boundary wrapping the Lagrangian submanifold $L$ we can extract the connected invariant defined in \eqref{eq: defn open GW} via \cite[Equation~(6.16)]{BBvG2}\footnote{We corrected \cite[Equation~(6.16)]{BBvG2} by an overall factor $-\ri$. See \cite[Appendix~A]{Yu23:BPS} for a discussion of this factor.}
\begin{equation}
	\label{eq: disconnected to connected}
	\sum_{\upbeta'} \sum_{g\geq 0} \OGW{g,(c),\upbeta'}{X,L,\msf{f}} ~ \hbar^{2g-1}\, Q^{\upbeta'} = -\ri \sum_{k=0}^c \frac{(-1)^k}{c} \frac{\cW_{(c-k,1^k)}(X,L,\msf{f})}{\cW_{\emptyset}(X,L,\msf{f})}
\end{equation}
under the change of variables $q=\re^{\ri \hbar}$ where we fixed a square root $\ri$ of $-1$.

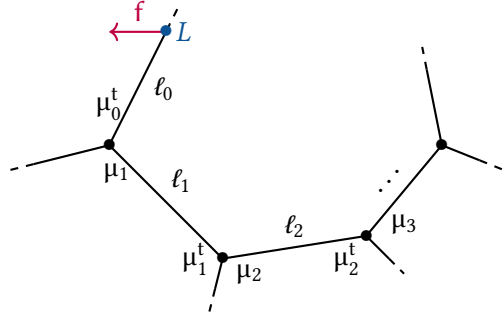
\begin{figure}
	\centering
	\begin{tikzpicture}[smooth, baseline={([yshift=-.5ex]current bounding box.center)}]%
		\draw[thick, dashed] (-0.1,-0.2) to (-0.2,-0.4);
		\draw[thick] (-0.2,-0.4) to (-1,-2);

		\draw[thick] (-1,-2) to (-2,-2.25);
		\draw[thick,dashed] (-2,-2.25) to (-2.4,-2.35);

		\draw[thick] (3.2,-1) to (3.4,-2);
		\draw[thick, dashed] (3.2,-1) to (3.14,-0.7);

		\draw[thick]  (3.4,-2) to (2.4,-3.2) to (0.5,-3.5) to (-1,-2);
		\draw[thick] (3.4,-2) to (3.9,-2.2);
		\draw[thick,dashed] (3.9,-2.2) to (4.2,-2.32);
		\draw[thick] (2.4,-3.2) to (2.7,-3.5);
		\draw[thick,dashed] (2.7,-3.5) to (2.9,-3.7);
		\draw[thick] (0.5,-3.5) to (0.4,-3.9);
		\draw[thick,dashed] (0.4,-3.9) to (0.325,-4.2);
		\node at (3.4,-2) {$\bullet$};
		\node at (2.4,-3.2) {$\bullet$};
		\node at (0.5,-3.5) {$\bullet$};
		\node at (-1,-2) {$\bullet$};
		\node[right] at (-0.55,-1.25) {$\ell_0$};
		\node[above] at (-0.05,-2.75) {$\ell_1$};
		\node[above] at (1.45,-3.35) {$\ell_2$};
		\node[above, rotate=50] at (2.9,-2.6) {$\cdots$};
		\node[above] at (-1,-1.8) {$\upmu_0^{\rt}$};
		\node[below] at (-0.9,-2.1) {$\upmu_1$};
		\node[left] at (0.45,-3.5) {$\upmu_1^{\rt}$};
		\node[right] at (0.55,-3.7) {$\upmu_2$};
		\node[left] at (2.45,-3.5) {$\upmu_2^{\rt}$};
		\node[right] at (2.6,-3.05) {$\upmu_3$};
		\node[above,purple] at (-0.6,-0.5) {$\msf{f}$};
		\draw[->,purple, thick] (-0.25,-0.5) to (-1,-0.5);
		\node[color=highlight] at (-0.25,-0.5) {$\bullet$};
		\node[right,highlight] at (-0.25,-0.5) {$L$};
	\end{tikzpicture}%
	\caption{Toric diagram of $(X,L,\msf{f})$ decorated with partitions.}
	\label{fig: tor skel}
\end{figure}

The generating series $\cW_{\upmu}(X,L,\msf{f})$ can be calculated algorithmically via the topological vertex method by performing a weighted sum over decorations of the toric diagram \cite{AKMV05:TopVert,LLLZ09:MathTopVert}. As we discussed in \Cref{sec: open geometric setup}, for our case at hand the toric diagram of $(X,L,\msf{f})$ will have a shape as displayed in \Cref{fig: tor skel}. We label the compact one dimensional torus orbit closures by $\ell_1,\ldots,\ell_k$ starting from the stratum $\ell_0$ intersected by $L$ and decorate each edge $\ell_i$ with a partition $\upmu_i$ as indicated in the figure. Then every edge $\ell_i$, $i\in\{0,\ldots,k\}$, contributes with a factor
\begin{equation*}
	(-1)^{(\deg N_{\ell_i}\defS) \cdot |\upmu_i|} q^{(\deg N_{\ell_i}\defS + 1)\cdot \upkappa_{\upmu_i}/2}
\end{equation*}
to \eqref{eq: disconnected open invariant} where $\upkappa_{\upnu}\coloneqq \sum_j \upnu_j (\upnu_j - 2j +1)$. Each trivalent vertex corresponding to a torus fixed point $\ell_i\cap \ell_{i+1}$ contributes
\begin{equation}
	\label{eq: two leg vertex formula}
	\cW_{\upmu_{i+1},\upmu_{i}^{\rt},\emptyset} \coloneqq s_{\upmu_{i+1}}\!\big(q^{\uprho+\upmu_{i}}\big) \, s_{\upmu_{i}^{\rt}}\!\big(q^{\uprho}\big) = q^{-\upkappa_{\upmu_{i}}/2} s_{\upmu_{i+1}}\!\big(q^{\uprho}\big) s_{\upmu_{i}}\!\big(q^{\uprho+\upmu_{i+1}}\big) = q^{\upkappa_{\upmu_{i+1}}/2} \sum_{\upnu} s_{\tfrac{\upmu_{i}^{\rt}}{\upnu}}\!\big(q^{\uprho}\big) \, s_{\tfrac{\upmu_{i+1}^{\rt}}{\upnu} }\!\big(q^{\uprho}\big)
\end{equation}
where by convention $\upmu_0=\upmu$ and $\upmu_{k+1}=\emptyset$ and $s_{\upalpha/\upbeta} \big(q^{\uprho+\upgamma}\big)$ denotes the skew Schur function $s_{\upalpha/\upbeta} \big(x_1,x_2,\ldots\big)$ evaluated at $x_j = q^{-j + \frac{1}{2} +\upgamma_j}$. As an application of \cite[Proposition~7.4]{LLLZ09:MathTopVert} we obtain
\begin{lem}
	\label{lem: outcome top vert}
	For $(X,L,\msf{f})$ the outcome of \Cref{constr: toric triple} we have
	\begin{equation*}
		\cW_{\upmu}(X,L,\msf{f})(q,Q) = \sum_{\upmu_1,\ldots,\upmu_k} Q^{\sum_{i=1}^k |\upmu_i| \cdot [\ell_i]} ~ \prod_{i=0}^k (-1)^{(\deg N_{\ell_i}\defS) \cdot |\upmu_i|} q^{(\deg N_{\ell_i}\defS + 1) \cdot \upkappa_{\upmu_i}/2} \cW_{\upmu_{i+1},\upmu_{i}^{\rt},\emptyset}\,.
	\end{equation*}
\end{lem}
In case the toric triple is a strip geometry the above sum of partitions can be carried out explicitly to yield a closed form solution for $\cW_{\upmu}(X,L,\msf{f})$ \cite{IKP06:VertexOnStrip}. Plugged into \eqref{eq: disconnected to connected} and \eqref{eq:log-open} this can be used to give a closed form solution for the logarithmic Gromov--Witten invariants of $(S\vertspace D_1+D_2)$. We illustrate this method with an explicit example in the following section.

\subsubsection{Extended example: \texorpdfstring{$\mr{dP}_3(0,2)$}{dP3(0,2)}}
\label{sec:dP3 0 2}
Let us consider $(\mr{dP}_3 \vertspace D_1+D_2)$ where $\mr{dP}_3$ is a del Pezzo surface which is the result of blowing up three points in $\bP^2$. This Looijenga pair was denoted by $\mr{dP}_3(0,2)$ in \cite{BBvG2}. If we write $E_1,E_2,E_3$ for exceptional divisors we take $D_1\in |H-E_1|$ and $D_2 \in |2H-E_2-E_3|$ smooth with transverse intersections. We use the following notation for a curve class in $\mr{dP}_3$:
\begin{equation*}
	\upbeta = d_0 (H-E_1-E_2-E_3) + d_1 E_1 +d_2 E_2 + d_3 E_3\,.
\end{equation*}
It was conjectured in \cite{BBvG2} and later proven in \cite{BS23:Quasitame} that the maximum contact logarithmic Gromov--Witten invariants of this geometry have the following closed form solution.

\begin{prop}
	\label{thm: dP3 0 2}
	\cite[Proposition~1.4]{BS23:Quasitame} For all effective curve classes $\upbeta$ in $\mr{dP}_3$ with $D_i \cdot \upbeta >0$, $i\in\{1,2\}$, and contact data $\mathbf{\hat{c}} = \begin{psmallmatrix} 0 & D_1 \cdot \upbeta & 0 \\ 0 & 0 & D_2 \cdot \upbeta \end{psmallmatrix}$ we have
	\begin{equation}
		\label{eq: Nlog dP3 0 2}
	\begin{split}
		&\sum_{g\geq 0} \hbar^{2g-1} ~ \LGW{g,\mathbf{\hat{c}},\upbeta}{\mr{dP}_3 \vertspace D_1+D_2}{(-1)^g \uplambda_g \,\mr{ev}_1(\mr{pt})} \\
		& \hspace{3em}= \frac{[d_1]_q [d_2+d_3]_q}{[d_0]_q [d_1 + d_2 + d_3 - d_0]_q} \qbinom{d_1}{d_0-d_3}_q \qbinom{d_1}{d_0-d_2}_q \qbinom{d_0}{d_1}_q \qbinom{d_1+d_2+d_3-d_0}{d_1}_q
	\end{split}
	\end{equation}
	as formal Laurent series in $\hbar$ under the identification $q=\re^{\ri \hbar}$.
\end{prop}
In this \namecref{thm: dP3 0 2} we use the notation
\begin{equation*}
	[n]_q \coloneqq q^{n/2}-q^{-n/2} \,, \quad
	\qbinom{n}{m}_q \coloneqq \begin{cases}
		\prod_{k=1}^{m}\frac{[n-m+k]_q}{[k]_q} & 0\leq m \leq n, \\
		0 & \text{otherwise}.
	\end{cases}
\end{equation*}
The statement was proven in \cite{BS23:Quasitame} by extracting the Gromov--Witten invariants from a quantum scattering diagram of a toric model of $(\mr{dP}_3 \vertspace D_1+D_2)$. The outcome of this calculation is a convoluted sum of products of $q$-binomial coefficients which with a lot of effort can be simplified to yield the above formula. We will give a new proof of this identity using the topological vertex method.

\begin{proof}[Proof of \Cref{thm: dP3 0 2}]
We start by applying \Cref{constr: toric triple} to $(\mr{dP}_3 \vertspace D_1+D_2)$. For this let us write $p_1,p_2,p_3$ for the three points in $\bP^2$ whose blow up gives $\mr{dP}_3$. Then $D_1$ is the strict transform of a line through $p_1$ and $D_2$ the strict transform of a conic through $p_2$ and $p_3$. We note that we can always choose the $\Gm^2$-action on $\bP^2$ in such a way that $D_1$ is a toric hypersurface and $p_1$ and $p_2$ are torus fixed points. Now with this choice of $\Gm^2$-action the blowups at $p_1$ and $p_2$ are toric as opposed to the one at $p_3$ since all points are assumed to be in general position. However, there is a deformation of $\mr{dP}_3$ to the blow up of $\bP^2$ in $p_1,p_2$ and a torus fixed point of $E_2$ leaving $D_1$ invariant. We denote this deformation by $\mr{dP}^{\prime}_3$ which is indeed toric as required. We display its fan in \Cref{fig: fan skel dP3} where we labelled each ray with its curve class.
\begin{figure}%
	\centering
	\begin{tikzpicture}[baseline=(current  bounding  box.center), scale=1.5]
		\draw[->,highlight] (0,-1.35) to (0,0);
		\node[left, color=highlight] at (0,-0.35) {$D_1=H-E_1$};
		\draw[->] (0,-1.35) to (1.35,-1.35);
		\node[above] at (1,-1.35) {$E_1$};
		\draw[->] (0,-1.35) to (-1.35,-1.35);
		\node[above] at (-1,-1.35) {$H-E_2-E_3$};
		\draw[->] (0,-1.35) to (-1.35,-2.7);
		\node[left] at (-0.7,-2.05) {$E_3$};
		\draw[->] (0,-1.35) to (0,-2.7);
		\node[right] at (-0.05,-2.35) {$E_2-E_3$};
		\draw[->] (0,-1.35) to (1.35,-2.7);
		\node[right] at (0.7,-2.05) {$H-E_1-E_2$};
		\draw (4,-0.3) to (4,-1) to (4.7,-1.7) to (5.7,-1.7) to (6.4,-1) to (6.4,-0.5);
		\node[right] at (4,-0.7) {$\ell_0$};
		\node[right] at (4.30,-1.25) {$\ell_1$};
		\node[above] at (5.2,-1.7) {$\ell_2$};
		\node[left] at (6.1,-1.25) {$\ell_3$};
		\draw[dashed] (4,-0.3) to (4,0);
		\draw[dashed] (6.4,-0.5) to (6.4,0);
		\draw (4,-1) to (3.5,-1);
		\draw[dashed] (3.5,-1) to (3,-1);
		\draw (4.7,-1.7) to (4.7,-2.2);
		\draw[dashed] (4.7,-2.2) to (4.7,-2.7);
		\draw (5.7,-1.7) to (5.7,-2.2);
		\draw[dashed] (5.7,-2.2) to (5.7,-2.7);
		\draw (6.4,-1) to (6.9,-1);
		\draw[dashed] (6.9,-1) to (7.4,-1);
		\draw[->,purple] (4,-0.3) to (3.5,-0.3);
		\node[color=purple, above] at (3.75,-0.3) {$\msf{f}$};
		\node[color=highlight] at (4,-0.3) {$\bullet$};
		\node[color=highlight, right] at (4,-0.3) {$L$};
	\end{tikzpicture}
	\caption{The fan of $\mr{dP}^{\prime}_3$ on the left and the toric diagram of $(X,L,\msf{f})$ on the right \cite[Figure~6.5]{BBvG2}.}
	\label{fig: fan skel dP3}
\end{figure}
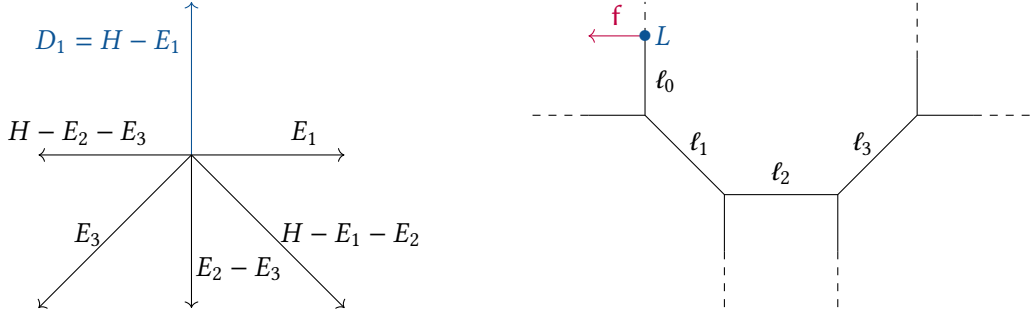%
Since $\mr{dP}^{\prime}_3$ is toric with $D_1$ a toric hypersurface, we can choose
\begin{equation*}
	X = \Tot \cO_{\mr{dP}^{\prime}_3} (-D_2) \lvert_{\mr{dP}^{\prime}_3 \setminus D_1}\,.
\end{equation*}
The toric diagram of $(X,L,\msf{f})$ is displayed in \Cref{fig: fan skel dP3} where we chose $L$ to intersect the torus invariant curve in class $H-E_2-E_3$. We note that this particular toric triple has already been studied by Bousseau, Brini and van Garrel in \cite[Section~6.3.1]{BBvG2}. Hence, from here on we basically follow their calculation.

Observe that each compact edge in the toric diagram corresponds to a ray in the fan. Hence, we can read off the degree of the normal bundles of $\ell_0,\ldots,\ell_3$ from \Cref{fig: fan skel dP3} to be $-1$, $-1$, $-2$ and $-1$ respectively. Thus, by \Cref{lem: outcome top vert} we have
\begin{align*}
	\cW_{\upmu}(X,L,\msf{f}) & = \sum_{\upmu_1,\upmu_2,\upmu_3} (-1)^{|\upmu|} \cW_{\upmu_1,\upmu^{\rt},\emptyset} \,\cdot\, (-1)^{|\upmu_1|} \cW_{\upmu_2,\upmu_1^{\rt},\emptyset} \,\cdot\, q^{-\upkappa_{\upmu_2}/2} \cW_{\upmu_3,\upmu_2^{\rt},\emptyset} \,\cdot\, (-1)^{|\upmu_3|} \cW_{\emptyset,\upmu_3^{\rt},\emptyset} ~ Q^{\sum_{i=1}^3 |\upmu_i| \cdot [\ell_i]}\\
	& = \sum_{\substack{\upmu_1,\upmu_2,\upmu_3 \\ \upnu_1,\upnu_2}} (-1)^{|\upmu|+|\upmu_1|+|\upmu_3|} ~ s_{\upmu^{\rt}}\big(q^{\uprho}\big) \, s_{\upmu_1^{\rt}}\!\big(q^{\uprho+\upmu}\big) \,  s_{\tfrac{\upmu_1}{\upnu_1}}\!\big(q^{\uprho}\big) \,  s_{\tfrac{\upmu_2}{\upnu_1}}\!\big(q^{\uprho}\big) \,  s_{\tfrac{\upmu_2}{\upnu_2}}\!\big(q^{\uprho}\big) \,  s_{\tfrac{\upmu_3}{\upnu_2}}\!\big(q^{\uprho}\big) \, s_{\upmu_3^{\rt}}\!\big(q^{\uprho}\big) ~ Q^{\sum_{i=1}^3 |\upmu_i| \cdot [\ell_i]}
\end{align*}
where to get the second line we used \eqref{eq: two leg vertex formula}. This can now be plugged into \eqref{eq: disconnected to connected} and simplified using the techniques developed in \cite{IKP06:VertexOnStrip}. These steps have already been carried out in \cite[Equation~(6.18--6.23)]{BBvG2} and result in
\begin{equation*}
	\begin{split}
		&\sum_{\upbeta'} \sum_{g\geq 0} \OGW{g,(c),\upbeta'}{X,L,\msf{f}} ~ \hbar^{2g-1}\, Q^{\upbeta'}\\
		&\hspace{3em}= \sum_{c_1,c_2,c_3} \ri (-1)^{c+c_1+c_3+1} \frac{1}{c [c_1]_q} \qbinom{c}{c_1 -c_2}_q \qbinom{c}{c_2 -c_3}_q \qbinom{c+c_3-1}{c_3}_q \qbinom{c_1}{c}_q ~ Q^{\sum_{i=1}^3 c_i [\ell_i]}\,.
	\end{split}
\end{equation*}
Extracting the coefficient of $Q^{\iota^{\cstar}(\upbeta - d_1 [\ell_0])}$ from the above formula we get
\begin{equation*}
	\begin{split}
		&\sum_{g\geq 0} \hbar^{2g-1} ~ \OGW{g,(d_1),\iota^{\cstar}(\upbeta - d_1 [\ell_0])}{X,L,\msf{f}}\\
		&\hspace{3em} =  \frac{\ri (-1)^{d_1+d_2+d_3+1} [d_1]_q}{d_1 [d_0]_q [d_1 + d_2 + d_3 - d_0]_q} \qbinom{d_1}{d_0-d_3}_q \qbinom{d_1}{d_0-d_2}_q \qbinom{d_0}{d_1}_q \qbinom{d_1+d_2+d_3-d_0}{d_1}_q
	\end{split}
\end{equation*}
which can be plugged into \eqref{eq:log-open} to yield the formula stated in \Cref{thm: dP3 0 2}.
\end{proof}

We stress that the above calculation is substantially easier than the one in \cite{BS23:Quasitame} using scattering diagrams. This hopefully illustrates the usefulness of our main result \Cref{thm: log open}.

\subsection{BPS integrality}
\label{sec: BPS integrality}
It is a general feature rather than a coincidence that the all-genus generating series of logarithmic Gromov--Witten invariants lifts to a rational function in $\re^{\ri \hbar}$ as we saw for instance in formula \eqref{eq: Nlog dP3 0 2}. This property is usually called BPS integrality.

\subsubsection{LMOV invariants}
In the context of open Gromov--Witten invariants of toric Calabi--Yau threefolds BPS integrality was first systematically studied by Labastida--Marino--Ooguri--Vafa \cite{OV00:KnotInvTopStr, LM01:PolyInvTorKnotTopStr, LMV00:KnotsLinksBranes, MV02:FramedKnotsLargeN}. Just recently Yu gave a proof of this phenomenon by carefully analysing the combinatorics of the topological vertex \cite{Yu23:BPS}.

To review Yu's result, let $(X,L,\msf{f})$ be a toric triple. For a curve class $\upbeta^{\prime}$ in $X$ and a winding profile $\mathbf{c}=(c_1,\ldots,c_n)$ along $L$ we will write
\begin{equation*}
	k  \vertspace  (\mathbf{c},\upbeta^{\prime})
\end{equation*}
for $k\in\bN$ if $k$ divides $c_i$ for all $i\in\{1,\ldots,n\}$ and there is a curve class $\upbeta^{\prime\prime}$ satisfying $k\upbeta^{\prime\prime}=\upbeta^{\prime}$.

Now for every tuple $(\mathbf{c},\upbeta^{\prime})$ we define
\begin{equation*}
	\invariantfont{LMOV}_{\mathbf{c},\upbeta^{\prime}}(\hbar) \in \bQ\llbracket \hbar^2 \rrbracket
\end{equation*}
by demanding that these formal power series satisfy
\begin{equation*}
	\sum_{g\geq 0} \hbar^{2g-2+n} ~  \OGW{g,\mathbf{c},\upbeta^{\prime}}{X,L,\msf{f}} \eqqcolon \prod_{i=1}^n \frac{2 \sin \tfrac{\hbar c_i}{2}}{c_i} ~ \sum_{k|(\mathbf{c},\upbeta^{\prime})} k^{n-1} \left(2 \sin \tfrac{\hbar k}{2}\right)^{-2}~\invariantfont{LMOV}_{\mathbf{c}/k,\upbeta^{\prime}/k}(k\hbar)\,.
\end{equation*}
Writing $\upmu$ for the M\"obius function one can check that
\begin{equation*}
	\invariantfont{LMOV}_{\mathbf{c},\upbeta^{\prime}}(\hbar) =  \left(2 \sin \tfrac{\hbar}{2}\right)^2 \prod_{i=1}^n \frac{c_i}{2 \sin \tfrac{\hbar c_i}{2}} ~ \sum_{k | (\mathbf{c},\upbeta^{\prime})} \frac{\upmu(k)}{k} \sum_{g\geq 0} (k\hbar)^{2g-2+n} ~  \OGW{g,\mathbf{c}/k,\upbeta^{\prime}/k}{X,L,\msf{f}}
\end{equation*}
is indeed a solution to the above defining equation.

\begin{thm}[LMOV integrality]
	\label{thm: LMOV integrality}
	\cite[Theorem 1.1]{Yu23:BPS} $\invariantfont{LMOV}_{\mathbf{c},\upbeta^{\prime}}$ lifts to a Laurent polynomial in $q=\re^{\ri \hbar}$ with integer coefficients.
\end{thm}

\subsubsection{BPS integrality of logarithmic invariants}
Now let $(S\vertspace D_1+D_2)$ be a logarithmic Calabi--Yau surface with boundary the union of two transversally intersecting smooth curves $D_1,D_2$. To a tuple $(\mathbf{\hat{c}},\upbeta)$, where $\upbeta$ is an effective curve class in $S$ satisfying $D_1 \cdot \upbeta \geq 0$ and $D_2 \cdot \upbeta > 0$ and $\mathbf{\hat{c}}$ is a contact condition of the form
\begin{equation*}
	\mathbf{\hat{c}} = \bigg(\raisebox{-0.9em}{\begin{minipage}{13.2em}$\hspace{0.2em}\begin{matrix}
				0 & \cdots & 0 & c_1 & \cdots & c_n & 0\\
				0 & \mathclap{\underbrace{\makebox[4.5em]{$\cdots$}}_{\text{$m$ times}}} & 0 & 0 & \cdots & 0 & D_2 \cdot \upbeta
			\end{matrix}$\end{minipage}}\bigg)
\end{equation*}
for some $c_1,\ldots,c_n >0$ and $m\leq n$, we associate a formal power series
\begin{equation*}
	\invariantfont{BPS}_{\mathbf{\hat{c}},\upbeta}(\hbar) \in \bQ\llbracket \hbar^2 \rrbracket
\end{equation*}
solving the recursive system
\begin{equation}
	\label{eq: LGW to BPS}
	\begin{split}
		& (-1)^{K_S \cdot \upbeta +1 }\sum_{g\geq 0} \hbar^{2g-1+n} ~ \LGW{g,\mathbf{\hat{c}},\upbeta}{S \vertspace D_1+D_2}{(-1)^g \uplambda_g \,\Pi_{i=1}^n \mr{ev}_i(\mr{pt})} \\
		&\hspace{3em}\eqqcolon\bigg(\prod_{i=1}^{n-m} \frac{2 \sin \tfrac{\hbar c_i}{2}}{c_i} \bigg) \bigg(\frac{1}{m!}\prod_{i=n-m+1}^{n} 2 \sin \tfrac{\hbar c_i}{2} \bigg) \bigg(2 \sin \tfrac{\hbar (D_2 \cdot \upbeta)}{2}\bigg) \sum_{k | (\mathbf{\hat{c}},\upbeta)} k^{n-1} \,  \left(2 \sin \tfrac{\hbar k}{2}\right)^{-2} ~ \invariantfont{BPS}_{\mathbf{\hat{c}}/k,\upbeta/k} (k \hbar) \,.
	\end{split}
\end{equation}
Again, one can check that
\begin{equation*}
	\begin{split}
		\invariantfont{BPS}_{\mathbf{\hat{c}},\upbeta} &= \left(2 \sin \tfrac{\hbar}{2}\right)^2 \bigg(\prod_{i=1}^{n-m} \frac{c_i}{2 \sin \tfrac{\hbar c_i}{2}} \bigg) \bigg(\frac{1}{m!}\prod_{i=n-m+1}^{n} \frac{1}{2 \sin \tfrac{\hbar c_i}{2}} \bigg) \frac{1}{2 \sin \tfrac{\hbar (D_2 \cdot \upbeta)}{2}} \\[0.3em]
		& \hspace{3em}\times\sum_{k | (\mathbf{\hat{c}},\upbeta)} \upmu(k) \, k^{m-1} \, (-1)^{K_S \cdot \upbeta /k +1} \sum_{g\geq 0} (k\hbar)^{2g-1+n} ~ \LGW{g,\mathbf{\hat{c}}/k,\upbeta/k}{S \vertspace D_1+D_2}{(-1)^g \uplambda_g \,\Pi_{i=1}^n \mr{ev}_i(\mr{pt})}
	\end{split}
\end{equation*}
is indeed a solution to this defining equation. Motivated by the last section we expect the following.

\begin{conj}
	\label{conj: BPS integrality}
	$\invariantfont{BPS}_{\hat{c},\upbeta}$ lifts to a Laurent polynomial in $q=\re^{\ri \hbar}$ with integer coefficients.
\end{conj}

Evidence for the \namecref{conj: BPS integrality} comes from the following immediate corollary of Yu's LMOV integrality (\Cref{thm: LMOV integrality}) and our logarithmic-open correspondence (\Cref{thm: log open}).

\begin{thm}
	\label{thm: BPS integrality}
	\Cref{conj: BPS integrality} holds if $(S \vertspace D_1 + D_2)$ satisfies \Cref{assumption: star}. Especially, if $(X,L,\msf{f})$ is the toric triple from \Cref{constr: toric triple} we have
	\begin{flalign*}
		&& \invariantfont{BPS}_{\mathbf{\hat{c}},\upbeta} = \invariantfont{LMOV}_{\mathbf{c},\upbeta^{\prime}}\,. && \qed
	\end{flalign*}
\end{thm}

Note that the pattern in which the sine factors enter the right-hand side of equation \eqref{eq: LGW to BPS} naturally suggests a generalisation to the case where we permit more than one point of contact with $D_2$. It remains to be seen whether BPS integrality persists in this case as well.

\renewcommand*{\bibfont}{\footnotesize}
\printbibliography\medskip

\end{document}